\documentclass[11pt]{amsart}
\usepackage{amsfonts,amssymb,amsmath,amsthm}
\usepackage{url}
\usepackage{enumerate}
\numberwithin{equation}{section}
\usepackage[a4paper,text={150mm,240mm},centering,headsep=5mm,footskip=10mm]{geometry}  

\usepackage{mathrsfs}  
\newcounter{CountAlpha}

\newtheorem{MainThm}[CountAlpha]{Theorem}
\newtheorem{Thm}{Theorem}[section]
\newtheorem{Lem}[Thm]{Lemma}
\newtheorem{Prop}[Thm]{Proposition}
\newtheorem{Cor}[Thm]{Corollary}
\newtheorem{Setup}[Thm]{Setup}

\theoremstyle{definition}
\newtheorem{Def}[Thm]{Definition}
\newtheorem{Constr}[Thm]{Construction}
\newtheorem{Rk}[Thm]{Remark}
\newtheorem{Ex}[Thm]{Example}
\newtheorem{Not}[Thm]{Notation}
\newtheorem{Obs}[Thm]{Observation}

\usepackage[usenames,dvipsnames]{color}
\usepackage{graphicx}
\usepackage[colorlinks=true,linkcolor=blue,pagebackref,citecolor=blue]{hyperref}
\usepackage{tikz}

\newcommand{\inv}{\mathrm{inv}}
\newcommand{\ord}{\mathrm{ord}}
\newcommand{\Diff}{\mathrm{Diff}}
\newcommand{\Spec}{\mathrm{Spec\,}}
\newcommand{\Sing}{\mathrm{Sing\,}}
\newcommand{\Dir}{\mathrm{Dir}}

\newcommand{\IDir}{\mathbb{D}\mathrm{ir}}

\newcommand{\wrt}{with respect to }
\newcommand{\WLOG}{without loss of generality }
\newcommand{\vin}{\rotatebox{90}{\ensuremath{\in}}}

\newcommand{\leftmapsto}{\mathrel{\reflectbox{\ensuremath{\longmapsto}}}}
\newcommand{\RSP}{regular system of parameters}

\newcommand{\simH}{\sim_{\cE}}
\newcommand{\simHx}{\sim_{\cE(x)}}
\newcommand{\simHxpi}{\sim_{\cE(x')}}
\newcommand{\IPH}{pair with history}
\newcommand{\IPsH}{pairs with history}
\newcommand{\IEH}{idealistic exponent with history}
\newcommand{\IEsH}{idealistic exponents with history}

\newcommand{\ini}{\operatorname{in}}
\newcommand{\gr}{\operatorname{gr}}

\newcommand{\resfield}{K}
\newcommand{\TC}{T}
\newcommand{\ITC}{\mathbb{\TC}}

\newcommand{\dBu}[2]{{#1}^{ ( #2 ) } }

\newcommand{\IA}{\mathbb{A}} 
\newcommand{\IC}{\mathbb{C}}	
\newcommand{\ID}{\mathbb{D}}	
\newcommand{\IE}{\mathbb{E}}	
	
\newcommand{\IQ}{\mathbb{Q}}	
\newcommand{\IR}{\mathbb{R}}	
\newcommand{\IZ}{\mathbb{Z}}	

\newcommand{\cD}{\mathcal{D}}
\newcommand{\cE}{\mathcal{E}}
\newcommand{\cF}{\mathcal{F}}
\newcommand{\cG}{\mathcal{G}}
\newcommand{\cH}{\mathcal{H}}
\newcommand{\cI}{\mathcal{I}}

\newcommand{\cO}{\mathcal{O}}
\newcommand{\cT}{\mathcal{T}}

\newcommand{\maxIdeal}{M}

\author{Bernd Schober}
\address{None (Hamburg, Germany)}
\email{schober.math@gmail.com}
\thanks{Partially supported by the DFG Emmy Noether Programme “Arithmetic over finitely generated fields” (Project number: 155362679) and by the DFG Research Fellowship "Resolution of singularities in positive characteristic: the case of small dimension or large characteristic" (Project number: 250489866).}


\begin{document}

\title{A polyhedral approach to the invariant of Bierstone and Milman}

\begin{abstract}
	Based on previous work by the author we deduce that the invariant introduced by Bierstone and Milman in order to give a proof for constructive resolution of singularities in characteristic zero can be determined purely by considering certain polyhedra and their projections.
\end{abstract}

\maketitle

%
%
%
%
%
%
%
%
%
%
%
%
%

\section{Introduction}
\label{Intro}

The goal of this article is to prove that the local invariant for resolution of singularities in characteristic zero by Bierstone and Milman \cite{BMheavy} can be determined using polyhedra associated to the singularity.

In his celebrated paper \cite{Hiro64} Hironaka proved the existence of resolution of singularities for arbitrary dimensional algebraic varieties over fields of characteristic zero.
The original proof is quite complicated and consists of more than 200 very technical pages.
Moreover, the result is not constructive.
Nowadays, there are quite accessible and constructive proofs available, which are based on Hironaka's work.
First results on this were published by Bierstone and Milman \cite{BMSimpleConstr}, \cite{BMheavy} and Villamayor \cite{OrlandoConstr}, \cite{OrlandoPatch}.
Other references are for example \cite{AnaSantiOrlando05,CutBook,EncinasHauser,Hauser,Kollar,Wlo}.
In recent years new techniques using stack-theoretic weighted blow-ups created even more efficiency in desingularization in characteristic zero,
e.g., \cite{AQ,ATW_weighted, McQuillan,Quek}.
In contrast to this, there are only results in small dimension, for special cases, or for weaker versions of resolution of singularities in positive and mixed characteristics. 
Since we focus in the present article on the characteristic zero situation, we do not go deeper into the details here
and only point out that a technical tool that turned out to be useful for investigations in positive characteristic is Hironaka's characteristic polyhedron of a singularity \cite{HiroCharPoly}
which arises as a suitable projection of the Newton polyhedron.

In this article we focus on the invariant of Bierstone and Milman \cite{BMheavy} 
(see also \cite{BMlight} where the special case of a hypersurface singularity is considered).
Given a scheme $ X $ of finite type over a field of characteristic zero and a point $ x \in X $,
the invariant $ \inv_X ( x ) $ has the form
\[ 
	\inv_X ( x ) = ( \nu_1, s_1;\, \nu_2, s_2; \, \ldots;\, \nu_t, s_t;\, \nu_{t + 1} ),
\]
where
$ \nu_1 = \nu_1 (x) = H_{X,x} $ is the Hilbert-Samuel function of $ X $ at $ x $, 
the entries $ s_r \geq 0 $ are integers counting exceptional divisors and
$ \nu_{r+1} > 0 $ are rational numbers representing certain higher multiplicities for $ r \geq 1 $ with additionally $ \nu_{t+1} \in \{ 0 , \infty \} $. 
Key properties of $ \inv_X(.)$ are that it determines a regular center for the upcoming blow-up, that it strictly decreases after the designated blow-up and that such a decrease can only happen finitely many times.

Since we focus on the construction of the invariant locally at a point, 
we pass to the completion $ \widehat\cO_{X,x} $ so that using the Cohen-Structure-Theorem for complete local rings, 
we can restrict the situation for the present article to
$ \widehat\cO_{X,x} \cong K[[T_1, \ldots, T_n]]/J $ for an ideal $ J $.
We neither address the question whether the deduced invariants arising from the polyhedra depend on passing to the completion nor do we discuss that the information extracted from the polyhedron is canonical for the singularity $ x \in X $. 
For Hironaka's characteristic polyhedron such kind of questions are discussed in \cite{CPcompl,CSCcompl} and \cite{CJSc}. 

In order to formulate our main result, we have to introduce a suitable language.
We will work with Hironaka's notion of pairs. 
A pair $ \IE = (J,b) $ consists of a non-zero ideal $ J \subset R := K[[T_1, \ldots, T_n]] $ and a positive rational number $ b \in \IQ_+ $.
The singular locus of $ \IE $ is defined as the locus of points where $ J $ has order at least $ b $.
This technicality is required to appropriately detect refined information on the singularity when restricting to smaller ambient spaces.  

Within our considerations we distinguish the regular system of parameters in a suitable way as 
$ (T_1, \ldots, T_n) = (u_1, \ldots, u_d; y_1, \ldots, y_s ) = (u,y) $
(cf.~condition $(*)$ in Theorem~\ref{Thm:BerndHiro} and Theorem~\ref{Thm:nuandsoon}). 
Given such a distinction $ (u,y ) $, the author introduced the projected polyhedron $ \Delta(\IE,u,y) $ (Definition~\ref{Def:PolyNonIntrinsic}) in
\cite{BerndPoly} (see also \cite{BerndThesis}).
Let us provide some intuition for this polyhedron: 
Suppose $ J = \langle f \rangle $ is principal.
Recall that $ \IE = (J,b) $ with some fixed $ b \in \mathbb{Q}_+ $.
Consider an expansion $ f = \sum_{(A,B)\in \IZ_{\geq 0}^{d+s}} \lambda_{A,B} u^A y^B $ for $ \lambda_{A,B} \in K $ and using multi-index notation, e.g., $ u^A = u_1^{A_1} \cdots u_d^{A_d} $. 
Then $ \Delta(\IE,u,y) $ is the closed convex subset of $ \IR_{\geq 0}^d $ containing all points of the set 
\[ 
	\left\{ \frac{A}{b-|B|} + \IR_{\geq 0}^d \ \bigg| \ \lambda_{A,B} \neq 0 \mbox{ and } |B| < b \right\} .
\] 
For example, for $ f = y_1^3 y_2 + u_1^7 u_2 y_1 + u_1^4 u_2^3 + u_1^6 u_2^4 y_2^3 + u_1 u_2^8 y_1 y_2 + u_2^9 y_2^5 + u_1^{16} u_2^6 $
and $ b = 4 $ the projected polyhedron $ \Delta (\IE,u,y) $ looks as follows
(where $ \delta := \delta(\IE,u,y) $ and $ d_1 := d_1 (\IE,u,y) $ are explained below):
\[
	\begin{tikzpicture}[scale=0.6]
	
	\draw[<->, thick] (0,4.7)--(0,0)--(7.7,0);
	
	\foreach \x in {0,...,7}
	\draw (\x,0.1) -- (\x,-0.1) node [below] {\small \x};
	
	\foreach \x in {1,...,4}
	\draw (0.1,\x) -- (-0.1,\x) node [left] {\small \x};

	\fill[lightgray] (1/2,9/2)--(1/2,8/2)--(4/4,3/4)--(7/3,1/3)--(15/2,1/3)--(15/2,9/2);
	\filldraw (7/3,1/3) circle (2.5pt);
	\filldraw (4/4,3/4) circle (2.5pt);
	\filldraw (6/1,4/1) circle (2.5pt);
	\filldraw (1/2,8/2) circle (2.5pt);
	\filldraw (4,3/2) circle (2.5pt);
	
	\draw[very thick] (1/2,4.6)--(1/2,8/2)--(4/4,3/4)--(7/3,1/3)--(7.6,1/3);
	
	\draw[thick, dashed] (7/4,0)--(0,7/4);
	\draw[thick] (0.1,7/4) -- (-0.75,7/4) node [left] {\small $ \delta $};
	\draw[thick] (7/4,0.1) -- (7/4,-0.75) node [below] {\small $ \delta $};

	\draw[thick, dotted] (1/2,4)--(1/2,0);
	\draw[thick] (1/2,0.1) -- (1/2,-0.75) node [below] {\small $ d_1 $};
	
\end{tikzpicture} 
\]

The numerical data coming from the polyhedron that we require is the following:
\[
\begin{array}{rclll}
	\delta ( \IE, u, y )  &:= & \inf \{  |v| = v_1 + \cdots + v_d \mid v \in \Delta (\IE,u,y)  \}
	& \mbox{(Definition~\ref{def:delta(Delta)})},
	\\
	d_i ( \IE, u , y  ) &:=& \inf \{ v_i \mid v = ( v_1, \ldots, v_d ) \in \Delta (\IE, u,y) \} 
	& \mbox{(Definition~\ref{Def:nu})},	 
\end{array}
\]
for $ i \in \{ 1, \ldots, d \} $.
If $ \Delta (\IE, u,y) \neq \emptyset $ and $ I \subseteq \{ 1, \ldots, d \} $, we define
\[
	\nu_I ( \IE, u, y ) := \delta ( \IE, u, y ) - \sum_{i \in I } d_i ( \IE, u, y ) 
\]
and otherwise  $ \nu_I ( \IE, u, y ) := \infty $. 

To encode the data coming from the exceptional divisors
we extend the notion of pairs to \emph{\IPsH}%
\footnote{After developing the notion of idealistic exponents with history in his thesis \cite{BerndThesis} the author recognized that it is a slight variant of \emph{NC-divisorial exponents} which Hironaka introduced in \cite{HiroThreeKey}.}
$ ( \IE, \cE ) $ (Definition~\ref{Def:IEH}),
where $ \cE $ is a map from which one can obtain the numbers $ d_i (\IE,u,y) $.
Then we define 
$
\nu ( \IE, \cE, u, y ) := 
\nu_{I_\cE}(\IE,u,y) $ (Definition~\ref{Def:IEH_number}),
where $ I_\cE \subseteq \{ 1, \ldots, d \} $
corresponds to the exceptional divisors containing $ x $.

In Construction~\ref{Constr:nu} we recall the definition of $ \inv_X(x) $ by Bierstone and Milman in terms of pairs with history. 
The entry $ \nu_{r+1}(x) $ of $ \inv_X(x) $ is determined from a pair $ (\cF_r , \cE_r) $. 
Here, $ (\cF_r , \cE_r) $ is constructed from $ (\cF_{r-1} , \cE_{r-1}) $ by suitably restricting to a smaller dimensional ambient space and taking the exceptional data into account.
In particular, $ (\cF_r, \cE_r ) $ is a pair on $ V (y_1, \ldots, y_{r-1}) $ for suitably chosen $ (y) = (y_1, \ldots, y_{r-1}, y_r) $
--- more precisely, the notion of maximal contact (Definition~\ref{Def:maximal}) is behind this suitable choice.

Let $ (u) = (u_1, \ldots, u_e) $ be a system extending $ (y) $ to a regular system of parameters for $ R $.
Our main result is the following connection to the polyhedra:

\begin{MainThm}[Theorem~\ref{Prop:ThmA2ndB}]
\label{MainThm:nuPurelyPoly}
	For every $ r \in \{1, \ldots, t\} $,
	the entry 
	$ \nu_{r+1}(x) $ is equal to 
	the invariant $ \nu ( \cF_r, \cE_r, u, y_r )  $ coming from the polyhedron $ \Delta (\cF_r,u,y_r) $.
\end{MainThm}

Let us point out that choosing  $ y_r $ such that $ V(y_r) $ has maximal contact with $ \cF_r $
is connected with finding a choice such that the polyhedron $ \Delta (\cF_r,u,y_r) $ becomes minimal with respect to inclusion (cf.~Theorem~\ref{Thm:BerndHiro}).
Since the minimizing process for the projected polyhedra can be introduced in any characteristic (cf.~\cite{BerndPoly}),
the results of the present article motivate the exploration into positive characteristics.
Nonetheless this is more involved and not straight-forward due to delicate phenomena arising, see for example \cite{CossartFewBad,Hauser2,HauserPerlega,Moh}.
In the forthcoming work \cite{AQS} Abramovich, Quek and the author will combine the perspective via weighted blow-ups \cite{ATW_weighted} and the present article to introduce an invariant for singularities in positive characteristic and describe its properties. 
In order to begin building the bridge,
we discuss a viewpoint using a weighted variant of the projected polyhedron in Observation~\ref{Obs:quasihomogeneous}. 

We point out that the considered polyhedra depend on a suitable choice for the regular system of parameters for $ R $ which will be modified along the process to make the polyhedra minimal with respect to inclusion. 
This is motivated by Hironaka's perspective via characteristic polyhedra~\cite{HiroCharPoly}. 
This distinguishes our approach from the results of Youssin \cite{Youssin1,Youssin2} who introduced a ``coordinate-free" version of Newton polyhedra to
construct a special filtration which then provides the center for the upcoming blow-up.
In particular, the regular system of parameter is not fixed and if there are parts of it fixed then without any reference to the generators of the ideal $ J $.

\subsection*{Outline of the article} 
In Section~\ref{Sec:Pairs} we recall the theory of pairs and their characteristic polyhedra.
In particular we introduce the numerical invariants deduced from the polyhedron which are essential for our main result. 
Following this, we recall the construction of the Bierstone-Milman invariant in the case where no exceptional divisors are present and prove
the main result in this setting in Section~\ref{Sec:NoExc}. 
Furthermore, we provide a look into the weighted perspective towards the end of the section.
In Section~\ref{Sec:History} we introduce the new notion of pairs and idealistic exponents with history as preparation for the general proof. 
After recalling the general construction of the Bierstone-Milman invariant at the beginning of Section~\ref{Sec:BMfull}, we deduce the main result in full generality.
Additionally, we discuss the behavior of the polyhedron in the different steps of the construction of the invariant.
Finally, we explain an approach to perform the construction in fewer steps by taking advantage of the added exceptional data. 

\subsection*{Acknowledgment} 
The results presented here are part of the author's thesis \cite{BerndThesis}.
He is grateful to his advisors Uwe Jannsen and Vincent Cossart for countless discussions on the topic and all their support.
Further, he thanks the Laboratoire de Math\'emathiques Versailles for their hospitality during several visits
as well as the anonymous referees for comments and questions which helped to improve the article.

%
%
%
%
%
%
%
%
%
%
%
%
%
%

\section{Pairs and their characteristic polyhedra}
\label{Sec:Pairs}

First, we recall the language of pairs and idealistic exponents.
This notion goes back to Hironaka \cite{HiroIdExp} (see also \cite{HiroThreeKey}) and has later been refined in different ways (e.g.~to basic objects \cite{AnaSantiOrlando05}, presentations \cite{BMheavy}, marked ideals \cite{Wlo}, or idealistic filtrations \cite{IFP1}).
We discuss the idealistic variants of the tangent cone, the directrix, as well as the coefficient ideal.
Further we give a brief introduction to the concept of the characteristic polyhedron of a pair, introduced by the author, see \cite{BerndPoly} (and also \cite{BerndThesis}).

Let
$ K $ be a field,
$ R = K[[T_1, \ldots, T_n]] $ or $ R = K[T_1, \ldots, T_n] $
and $ M = \langle T_1, \ldots, T_n \rangle $.
We denote by $ x \in \Spec(R) $ the closed point corresponding to the maximal ideal $ M $.

\begin{Def}
\label{Def:idexp_sing_perm}
	Let $ J \subset R $ be a non-zero ideal and $ b \in \IQ_+ $ be a positive rational number.
	We call $ \IE := (J,b) $ a  {\em pair on $ R $}. 
	Moreover, we define: 
	\begin{enumerate}[(1)]
		\item 
		For a prime ideal $ P \subset R $ the {\em order of $ \IE $ at $ P $} is defined as 
		as 
		\[  
			\ord_P (\IE) := 
			\begin{cases}
				\dfrac{\ord_P (J) }{b}, 	
				& 
				\mbox{if } \ord_P (J)  \geq b, 
				\\
				0,
				&
				\mbox{else},
		\end{cases}
		\] 
		where
		$ \ord_P (J) := \sup \{ d \in \IZ_{ \geq 0 } \cup \{\infty\} \mid J \subseteq P^{(d)} \} $ is the order of $ J $ at $ P $.
		
		\item 	
		The {\em singular locus of $ \IE $} (sometimes also called {\em (co-)support})
		is defined as 
		\[ 
			\Sing(\IE)= \{ P \in \Spec(R)  \mid \ord_P ( J ) \geq b \}. 
		\]
	
		\item 
		\label{It:permissible}
		A subscheme $ D \subset Z := \Spec (R ) $ 
		is called {\em permissible for $ \IE $}
		if $ D $ is regular and $ D \subseteq \Sing(\IE) $. 
		The corresponding blow-up
		$ \pi \colon Z' := B\ell_D (Z) \to Z  $ with center $ D $ is called permissible for $ \IE $.
		The {\em transform of $ \IE $} 
		at a point $ x' \in Z' $ (resp.~in an affine open $ U' \subset Z' $)
		is defined 
		as $ \IE' = (J', b) $, 
		where $ J' $ is determined by $ J \cO_{Z',x'} = J' H^b $
		(resp.~$ J \cO_{Z',U'} = J' H^b $)
		and $ H $ is the ideal locally defining the exceptional divisor.
\end{enumerate}
\end{Def}

	A {\em desingularization of a pair} is a finite sequence of permissible blow-ups 
	such that the singular locus of the final transform is empty. 
	
	In order to understand a given pair better,
	it is reasonable to identify two pairs which show the same behavior along permissible blow-ups.  
	This leads to an equivalence relation $ \sim $ on the set of pairs on $ R $.
	An {\em idealistic exponent} $ \IE_\sim $ is the equivalence class\footnote{In some literature, pairs are sometimes called idealistic exponents, e.g.~\cite{HiroThreeKey}. In order to avoid confusion when coming to results and the dependence on the choice of a representative of the equivalence class, 
		we use the original terminology of \cite{HiroIdExp}.}
	of a given pair $ \IE $. 
	The precise definition is quite technical 
	(involving local sequences of permissible blow-ups and allowing extensions of the ambient space $ Z = \Spec(R) $ 
	of the form $ Z \times_{\mbox{\footnotesize Spec} (\IZ)} \Spec (\IZ[\xi_1, \ldots, \xi_a]) $).
	As it will not be used in this article, 
	we refer to the literature for details, e.g.~\cite{HiroThreeKey} or \cite[Definition~1.5]{BerndPoly}.
	Here, we only recall results and constructions, which we will use. 
	Of course, the definition of the equivalence $ \sim $ plays an essential role in their proofs.

If $ \IE_1 = (J_1, b_1) $ and $ \IE_2 = (J_2, b_2) $ are two pairs on $ R $
and if we set $ c := \operatorname{lcm}(b_1,b_2) $ and $ c_1 := \frac{c}{b_1} $, $ c_2 := \frac{c}{b_2} $,
then the intersection $ \IE_1 \cap \IE_2 $ is defined by
\[
	\IE_1 \cap \IE_2 := (J_1^{c_1} + J_2^{c_2}, c ).
\]

\begin{Prop}
\label{Prop:basicandmore}
	Let $ \IE = (J,b) $ and $ \IE_i = (J_i,b_i) $, $ i \in \{1, 2 \}$, be pairs on $ R $.
	\begin{enumerate}[(1)]
		\item\label{It:Power}
		For every $ a \in \IZ_+ $,
		we have $ (J^a, a b) \sim (J,b) $.
					
		\item\label{It:intersec}
		If $ m \in \IZ_+ $ with $ b_1 \mid m $ and $ b_2 \mid m $, 
		then 
		$
			(J_1,b_1) \cap (J_2,b_2) \sim (J_1^{\frac{m}{b_1}} + J_2^{\frac{m}{b_2}}, m).
		$
			
		\item\label{It:sim_intersec}
		If $ \IE_1 \sim  \IE_2 $, then $ \IE_1 \cap \IE \sim \IE_2 \cap \IE $.
		
		\item\label{It:intersec_bu}
		Let $ \pi \colon Z' \to Z = \Spec(R) $ be a blow-up which is permissible for $ \IE_1 $ and $ \IE_2 $.
		We have $ (\IE_1 \cap \IE_ 2)' \sim \IE'_1 \cap \IE'_2 $.
		
		\item\label{It:Numerical} 
		If $ \IE_1 \sim \IE_2 $, then $ \ord_P (\IE_1) = \ord_P (\IE_2) $, for every prime ideal $ P \subset R $.
		In particular, we get $ \Sing(\IE_1) = \Sing (\IE_2) $ in this case.
		
		\item\label{It:Diff}
		Let  $ m \in \IZ_{\geq 0} $ and let 
		$ \cD $ be a left $ R $-submodule of the absolute differential operators $ \Diff^{\leq m}_{\IZ}(R) $ of $ R $ on itself of order at most $ m $.
		Then,  $ (J, b) \sim  (\cD J, b - m) \cap (J,b) $.
		
	\end{enumerate}
\end{Prop}

\begin{proof}
	Proofs for the properties \eqref{It:Power}--\eqref{It:intersec_bu} can be found in \cite[Lemma~1.1.8]{BerndThesis}.
	Part \eqref{It:Numerical} is a consequence of the {\em Numerical Exponent Theorem}
	\cite[Theorem 5.1]{HiroThreeKey}
	and \eqref{It:Diff} follows from the {\em Diff Theorem} \cite[Theorem 3.4]{HiroThreeKey}.
\end{proof}

The graded ring of $ R $ at $ M $ is defined as 
$ \gr_M (R) := \bigoplus_{a \geq 0} M^a/ M^{a+1} $.
We have $ \gr_M(R) \cong K[\mathcal{T}_1, \ldots, \mathcal{T}_n] $,
where $ \mathcal{T}_i $ corresponds to the image of $ T_i $ in the graded ring. 
The {\em initial form} at $ M $ of a non-zero element $ f \in R $ is given as the image of $ f $ in $ M^{\ord_M(f)} / M^{\ord_M(f) +1} $ and is denoted by $ \ini_M(f) $,
i.e., 
$ \ini_M(f) = f \mod M^{\ord_M(f)+1} $.
Using $ \ini_M(0) := 0 $,
the {\em initial ideal} of an ideal $ J \subset R $ (at $ M $) is 
defined as $ \ini_M(J) := \langle \ini_M(f) \mid f \in J \rangle \subseteq \gr_M(R) $.
In the setting of pairs these notions adapt as follows:

\begin{Def}
	Let $ \IE = (J, b) $ be a pair on $ R $ such that $ x \in \Sing(\IE) $.
	The {\em $ b $-initial form} of $ f \in J $ (\wrt $ \maxIdeal $) is defined as 
	\[
			\ini_M (f,b) := 
			\begin{cases}
				f \mod M^{ b + 1 }, 	
				&	 
				\mbox{if } 
				b \in \IZ_+ , 
				\\ 
				0 ,					
				&	 
				\mbox{if } 
				b \notin \IZ_+ .
			\end{cases}		
	\]
	We define the {\em tangent cone of $ \IE $ at $ M $}  
	as the cone $ \TC_M (\IE) \subset \Spec(\gr_M(R)) \cong \IA_K^n  $ 
	determined by the homogeneous ideal $ \ini_M (J,b) \subset \gr_M (R) $, 
	where
	\[
		\ini_M (J,b) := \ini_M (\IE) 
		:= 
		\begin{cases}
			J \mod \maxIdeal^{ b + 1 } 
			= 
			\langle \ini_M(f, b) \mid f \in J \rangle, 	
			&	 
			\mbox{if } 
			b \in \IZ_+, 
			\\ 
			\langle 0 \rangle,					
			&	 
			\mbox{if } 
			b \notin \IZ_+ .
		\end{cases}
	\]
\end{Def}

The tangent cone gives a first approximation for the singularity of $ \IE $ at $ M $. 
In the context of singularities, Hironaka introduced the notion of the directrix of a homogeneous ideal, 
which provides refined information on the homogeneous ideal. 
This notion was adapted for pairs and idealistic exponents by the author in \cite[Section 2]{BerndPoly}: 

\begin{Def}
	\phantomsection
	\label{Def:directrix_both}
	\begin{enumerate}[(1)]
		\item 	
		\label{It:DefDirI}
		Let $ S = K[\cT_1, \ldots, \cT_n] $ and let $ I \subset S $ be a non-zero, homogeneous ideal.
		Let $ C := C(I) := \Spec(S/I) \subset \IA_K^n $ be the cone corresponding to $ I $.
		The {\em directrix of $ C $} is the maximal sub-vector space $ \Dir(C) $ 
		of $ \IA_K^n $,
		which leaves the cone $ C $ stable under translation,
		i.e., such that $ C + \Dir(C) = C $. 
		
		Thus, $ \Dir(C) $ corresponds to the minimal sub-vector space $ U = \bigoplus_{j=1}^r K Y_j \subseteq S_1 = \bigoplus_{i=1}^n K \cT_i $
		(which is generated by homogeneous elements $ Y_1, \ldots, Y_r \in S_1 $ of degree one) such that 
		$ ( I \cap K[Y_1, \ldots, Y_r])\cdot S = I $;
		i.e.,
		$ ( Y_1, \ldots, Y_r ) $ is a smallest list of variables such that 
		there is a system of generators of $ I $ contained in $ K[Y_1, \ldots, Y_r] $. 
		We also say $ ( Y ) = ( Y_1, \ldots, Y_r ) $ defines the directrix and we implicitly assume that $ r $ is minimal.
		We call $ I\Dir ( C ) := \langle Y_1, \ldots, Y_r \rangle $ the {\em ideal of the directrix of $ C $}.
	
		\item
		\label{It:DefDirIE}
		Let $ \IE = (J,b) $ be a pair on $ R $ such that $ x \in \Sing (\IE) $.
		The {\em directrix of $ \IE $ at $ M $} is defined as  
		the directrix of the tangent cone of $ \IE $ at $ M $,
		$  \Dir_M (\IE) := \Dir (\TC_M (\IE)) $.
	\end{enumerate}
\end{Def}

There exists also the notion of the {\em ridge} (in French {\em fa\^ite})
of a cone \cite{GiraudEtude,GiraudPos,ComputeRidge},
which is a generalization of the directrix.
It is more suitable when working over fields of positive characteristic, see loc.~cit. or \cite{CPS}, 
but it has the drawback that it is an additive group scheme and not a vector space as the directrix.  
For more details on its idealistic variant, we refer to \cite[Section~2]{BerndPoly}.

\begin{Def}
	Let $ \IE = (J,b) $ be a pair on $ R $ with  $ x \in \Sing ( \IE ) $.
	We introduce the following pairs on $ \gr_M(R) \cong K[\cT_1, \ldots, \cT_n] $:
	\begin{center}
		\begin{tabular}{ll}
			$ \ITC_M ( \IE ) = (\, \ini_M ( \IE ),\, b \,) $ & \em (idealistic tangent cone of $ \IE $ at $ M $), \\[5pt]
			$ \IDir_M ( \IE ) = (\, I\Dir_M ( \IE ),\, 1 \,) $  & \em (idealistic directrix of $ \IE $ at $ M $).
		\end{tabular} 
	\end{center}	
\end{Def}

Note that for $ \IE_1 = (J_1,b_1) $ and  $ \IE_2 = (J_2,b_2) $ with $ x \in \Sing ( \IE_1 \cap \IE_2 ) $, we have:
\begin{center}
	\begin{tabular}{l}
		$ \ITC_M ( \IE_1 \cap \IE_2 ) =  \ITC_M ( \IE_1 ) \cap  \ITC_M (\IE_2 ) = (\, \ini_M ( \IE_1 ),\, b_1 \,) \cap (\, \ini_M ( \IE_2 ),\, b_2 \,) $,
		\\[5pt]
		$ \IDir_M ( \IE_1 \cap \IE_2 ) = \IDir_M ( \IE_1 ) \cap  \IDir_M ( \IE_2 ) = (\, I\Dir_M ( \IE_1 )+ I\Dir_M ( \IE_2 ),\, 1 \,)  $.
	\end{tabular} 
\end{center}

As the following result shows, the adjective ``idealistic" is justified, as equivalent pairs lead to equivalent idealistic tangent cones and directrices. 
Furthermore, we can connect the two objects of the previous definition. 

\begin{Prop}
	\phantomsection
	\label{Prop:Dir}
	\begin{enumerate}[(1)]
		\item
		Let $ \IE $ be a pair on $ R $ such that $ x \in \Sing ( \IE) $.
		We have: 
		\begin{enumerate}[(a)]
			\item 
			$ \IDir_M ( \IE ) \sim \IDir_M ( \IE ) \cap \ITC_M(\IE) $,
			
			\item 
			$ \Dir_M (\IE) =  \Sing (\IDir_M(\IE)) \subseteq \Sing (\ITC_M(\IE)) $,
			and
			
			\item
			if $ \operatorname{char}(K) = 0 $ or $ b < \operatorname{char} (K) $,
			then
			$ \IDir_M ( \IE ) \sim \ITC_M(\IE) $;
			in particular, the inclusion in (b) is an equality.
			 
		\end{enumerate} 
		\item\label{It:Dir2}
		Let $ \IE_1 \sim \IE_2 $ be two equivalent pairs on $ R $ such that $ x \in \Sing ( \IE_1 ) = \Sing ( \IE_2 ) $.
		We have 
		$ \ITC_M (\IE_1) \sim  \ITC_M (\IE_2) $
		and
		$ \Dir_M (\IE_1) = \Dir_M (\IE_2) $, 
	so 
	$  \IDir_M (\IE_1) = \IDir_M (\IE_2) $.
	\end{enumerate}
\end{Prop}

\begin{proof}
	Parts (1)(a) and (b) are consequences of
	\cite[Proposition~2.13]{BerndPoly},
	while (c) is \cite[Corollary~2.14]{BerndPoly}.
	The second part follows from \cite[Proposition~2.16]{BerndPoly}.
\end{proof}

Another useful tool for the local study of singularities (in characteristic zero)
is the coefficient ideal \wrt a closed subscheme of maximal contact.
We now recall its variant in the idealistic setting.
Note that we introduce the coefficient pair \wrt any $ W = V( z_1, \ldots, z_s) $ containing $ x $ such that $ (z) $ can be extended to a system of parameters for $ R $,
cf.~\cite[Section~5]{BerndPoly}.

Recall that $ R = K[[T_1, \ldots, T_n]] $ or $ R = K[T_1, \ldots, T_n] $.

\begin{Def}
\label{Def:IdCoeffExp}
	Let $ \IE = (J,b) $ be a pair on $ R $ such that $ x \in \Sing(\IE) $.
	Let $ (u, z ) = (u_1, \ldots, u_d; z_1, \ldots, z_s ) $ be a regular system of parameters for $ R $.
	The {\em coefficient pair $ \ID (\IE, u, z )  $ of $ \IE $ \wrt $ (z) $}
	is defined as the pair on $ R/\langle z \rangle $,
	constructed as follows:
	Any element $ f \in J $ can be written as
	\[ 
		f = f (u,z) = \sum_{B \in \IZ^s_{ \geq 0 } : |B|<b } f_B (u)\, z^B + h,
	\]
	where we have $ h \in \langle z \rangle^b $ and 
	$ f_B(u) \in K[[u]] $ 
	if $ R = K[[T_1, \ldots, T_n]] $,
	resp.~$ f_B(u) \in K[u] $ if $ R = K[T_1, \ldots, T_n] $. 
	We define  
	\[ 
		\ID(f,u,z)  := \bigcap\limits_{|B| < b} ( f_B (u) , \, b -|B| )
		\ \ \ 
		\mbox{ and } 
		\ \ \
		\ID (\IE, u, z) := \bigcap_{ f \in J } \ID(f, u, z)  .
	\]
	Here, we abuse notation by using the symbol $ f_B(u) $ also for the image of $ f_B(u) $ in $ R/\langle z\rangle $. 
\end{Def}

Before coming to results on coefficient pairs, we recall the concept of maximal contact.
Classical references for this are \cite{AHV,GiraudMaxZero}, for example.

\begin{Def}
	\label{Def:maximal}
	Let $ \IE = (J,b) $ be a pair on $ R $ with $ x \in \Sing ( \IE ) $.
	Let $ (z) = (z_1, \ldots, z_s) $ be elements
	which can be extended to a regular system of parameters for $ R $.
	We say $ V ( z_1, \ldots, z_s ) $ 
	has {\em maximal contact with $ \IE $ (at $ M $)} 
	if
	$ 
		\IE \sim (z, 1) \cap \IE.
	$ 
\end{Def}

\begin{Prop}
	\phantomsection
	\label{Prop:moreCoef}
	\begin{enumerate}[(1)]
		\item\label{It:equiv_z_E_DE} 
		Let $ \IE $ be a pair on $ R $ such that $ x \in \Sing ( \IE ) $
		and further let $ ( u ,z ) =  (u_1, \ldots, u_d; z_1, \ldots, z_s ) $ be a regular system of parameters for $ R $.
		\begin{enumerate}[(a)]
			\item\label{It:291a}
			We have 
			$ 
				(z,1) \cap \IE \sim (z,1) \cap \ID (\IE , u, z) .
			$ 
			
			\item\label{It:291b} 
			Let $ (y) = ( y_1, \ldots, y_s ) $ be another system extending $ ( u ) $ to a  regular system of parameters for $ R $.
			If $ (z,1) \cap \IE \sim (y,1) \cap \IE $,
			then 
			$ 
				\ID (\IE , u, z) \sim \ID (\IE , u, y) .
			$ 
		\end{enumerate}
		
		\item\label{It:sim_DE}
		Let $ \IE_1 \sim \IE_2 $ be two equivalent pairs on $ R $ with $ x \in \Sing ( \IE_1 ) = \Sing ( \IE_2 ) $.
		We have 
		$ 
			\ID (\IE_1 , u, z) \sim \ID (\IE_2 , u, z) .
		$ 
		
		\item\label{It:equiv_E_sim}
		Assume that $ char (\resfield) = 0 $ or $ b < char (\resfield) $.
		Then there exists a choice 
		$ (u,y) = (u_1, \ldots, u_e, y_1, \ldots, y_r) $ 
		for the parameters for $ R $
		such that 
		$ 
			\IE \sim (y,1) \cap  \ID (\IE, u, y ) 
		$ 
		and	$ (Y) = ( Y_1, \ldots, Y_r ) $ defines the directrix $ \Dir_M(\IE) $
		where  $ Y_j := y_j \mod M^2 $ for $ j \in \{ 1, \ldots, r \} $.
	\end{enumerate}
\end{Prop}

\begin{proof}
	The first part follows from \cite[Corollary~5.4 and Proposition~5.5]{BerndPoly}
	and part \eqref{It:sim_DE} from \cite[Theorem~5.3]{BerndPoly}.
	The last part is a  consequence of \cite[Lemma~5.6]{BerndPoly}.
	(Let us point out that our proofs there hold word-by-word in the case of a polynomial ring.)
\end{proof}

%
%
%
%
%
%
%
%
%
%
%
%
%
%

The aim of the present article is to deduce the invariant of Bierstone and Milman only by considering polyhedra.
Thus, let us recall the author's notion of characteristic polyhedra for pairs and their properties from \cite{BerndPoly}.

\underline{From now on, we fix $ R = K[[T_1, \ldots, T_n]] $.} 
Let $ \IE  = ( J, b ) $ be a pair on $ R $ such that $ x \in \Sing (\IE) $
and let $ (u) = ( u_1, \ldots, u_e)$ be a fixed system of elements in $ M $ which can be extended to a regular system of parameters for $ R $.
We consider various choices for systems $ (y) = (y_1, \ldots, y_r) $ extending $ (u) $ as mentioned.  

Let $ ( f ) = ( f_1, \ldots, f_m ) $ be a system of generators for $ J $.
By \cite[Proposition~2.1]{CPmixed}, 
every element $ f_i $ has a finite expansion 
%
\[	
	f_i = \sum_{ (A,B) \in \IZ^{e+r}_{ \geq 0 } } C_{A,B,i} \, u^A \, y^B
	\quad \quad \mbox{ with } \quad 
	C_{ A, B,i } \in R^\times \cup \{ 0 \}
\]
and the set $ \{ (A,B) \mid C_{A,B,i} \neq 0 \} $ is unique if we assume that the number of elements in this set is minimal. 
For our purposes, the uniqueness will not be needed.  
%

\begin{Def}
\label{Def:PolyNonIntrinsic}
For the given data, we define the {\em Newton polyhedron
		$ \Delta^N (\IE, u, y)  $ 
		of $ \IE = (J,b) $ \wrt $ ( u, y ) $}
	as the convex hull in $ \IR^{e+r}_{ \geq 0 } $ of the set
	\[ 
	\left\{ \frac{ (A,B) }{ b  } + \IR^{e+r}_{ \geq 0 }  \;\bigg|\; 1 \leq i \leq m ,\, C_{ A, B, i }  \neq 0 ,\, | B | \leq b \right\}.
	\] 
	
	The {\em (projected) polyhedron 
$ \Delta (\IE, u, y)  $ 
of $ \IE = (J,b) $ \wrt $ ( u, y ) $} 
is defined as the smallest closed convex subset of $ \IR^e_{ \geq 0 } $ containing all elements of the set
\[ 
	\left\{ \frac{ A }{ b - |B|  } + \IR^e_{ \geq 0 }  \;\bigg|\; 1 \leq i \leq m , \, C_{ A, B, i }  \neq 0 ,\, | B | < b \right\}.
\] 
\end{Def}

If $ \IE' $ is another pair on $ R $ with $ x \in \Sing ( \IE' ) $,
then $ \Delta (\IE \cap \IE', u, y) \subset \IR^e_{ \geq 0 } $ is the smallest closed convex subset containing $ \Delta (\IE, u, y) $ and $ \Delta (\IE', u, y) $.

In \cite[Proposition~3.3]{BerndPoly} it is shown that the polyhedron $ \Delta (\IE,u,y) $ is a projection of the Newton polyhedron $ \Delta^N (\IE,u,y) $.
Furthermore, both polyhedra are independent of the chosen set of generators $ ( f ) = ( f_1, \ldots, f_m ) $ by \cite[Lemma~3.2 and Corollary~3.4]{BerndPoly}. 

The polyhedron $ \Delta (\IE,u,y) $ is not necessarily invariant under the equivalence relation $ \sim $.
Nonetheless, we see later how we can get intrinsic information on the idealistic exponent by using it.
Since it motivates our further investigation, 
let us briefly recall an example, where the polyhedron changes under $ \sim $.

\begin{Ex}[{\cite[Example~3.6]{BerndPoly}}]
\label{Ex:PolyNotUnique}
	Let $ K = \IC $, $ d \in \IZ_+ , d \geq 2 $.
	Consider the following two pairs on $ \IC[[u_1, u_2, y_1, y_2]] $
	\[ 
	\begin{array}{l}
 		\IE_1 = (y_1^d - u_1^{d-1} u_2^{d-1},\, d) \;\cap\; (y_2,\, 1)  	\\[5pt]
 		\IE_2 = (y_1^d - u_1^{d-1} u_2^{d-1},\, d) \;\cap\; (y_2^{d-1}-u_1^{d-2} u_2^{d-1},\, d-1 )
	\end{array} 
	\]
	As shown in loc.~cit., $ \IE_1 \sim \IE_2 $, 
	but we have
	$ \Delta(\IE_1, u, y) \neq \Delta(\IE_2, u, y) $.
	More precisely, 
	$ v := \left(\frac{d-2}{d-1},\, 1\right) \in  \Delta(\IE_2, u, y) $
	and $ v \notin \Delta(\IE_1, u, y) $.
\end{Ex}

An important invariant of the singularity of $ \IE $ at $ x $ is the order of the coefficient pair \wrt a system $ ( y ) $ which determines $ \Dir_M ( \IE ) $.
Using the following definition this can be recovered from the polyhedron $ \Delta (\IE, u, y) $.

\begin{Def}
\label{def:delta(Delta)}
	Let  $ \Delta \subset \IR^e_{ \geq 0 } $ be any subset.
	We define 
	\[ 
		\delta ( \Delta ) := \inf \{ \, |v| = v_1 + \cdots + v_e \mid v = (v_1, \ldots, v_e) \in \Delta  \, \}.
	\]
	If $ \Delta = \Delta (\IE, u, y) $ and $ x \in \Sing ( \IE) $, 
	then we set 
	$ 
		\delta (\IE, u, y) := \delta ( \Delta (\IE, u, y) ) .
	$ 
\end{Def}

\begin{Prop}[{\cite[Lemma 5.2 and Proposition 4.6]{BerndPoly}}]
\label{Prop2.4}
	Let $ \IE $ be a pair on $ R  $ with $ x \in \Sing ( \IE ) $, and 
	let $ ( u ,y ) $ be a {\RSP} for $ R $.
\begin{enumerate}[(1)]
	\item
	The polyhedron $ \Delta (\IE,u,y) $ coincides with the Newton polyhedron of the coefficient pair $ \ID (\IE, u, y)  $.
	In particular, we have
	$
		 \delta (\IE, u, y )
		= \ord_M ( \ID (\IE, u, y)  ) $.
	\item
	For a pair $ \IE_2 $ which is equivalent to $ \IE_1 := \IE $, we have 
	$ 
	 \delta ( \IE_1, u, y)  = \delta ( \IE_2, u, y ) .
	$ 
	\item 
	\label{It:delta_ord_coeff}
	If $ ( u, z ) $ is another choice for the {\RSP} fulfilling the additional property 
	$ (z,1) \cap \IE \sim (y,1) \cap \IE $, then
	$ 
		\delta ( \IE, u, y ) = \delta ( \IE, u, z ) .
	$ 
	In particular, this implies that this number is independent of the choice of a maximal contact subscheme if we fix $ (u) $.
\end{enumerate}
\end{Prop}


\begin{Def}
	Let $ \IE = (J, b) $ be a pair on $ R $ such that $ x \in \Sing ( \IE ) $, and let $ ( u ) = (u_1, \ldots, u_e ) $ be a system of regular elements that can be extended to a {\RSP} of $ R  $.	
	We define
	\[ 
		\Delta (\IE, u) := \bigcap_{(y)} \Delta (\IE ,u, y),
	\]
	where the intersection ranges over all systems $ ( y ) $ extending $ ( u ) $ to a {\RSP} of $ R $.
	We call $ \Delta (\IE, u) $ the {\em characteristic polyhedron of the pair $ \IE $ \wrt $ ( u ) $}.
\end{Def}

\begin{Thm}[{\cite[Theorem~3.15]{BerndPoly}}]
\label{Thm:BerndHiro}
	Let $ \IE = (J, b) $ be a pair on $ R = K[[T_1, \ldots, T_n]] $ such that $ x \in \Sing ( \IE ) $ and
	let $ ( u, y ) = ( u_1, \ldots, u_e; y_1, \ldots, y_r ) $ be a {\RSP} for $ R $ such that the following condition holds:
	\[
	(Y) = (Y_1, \ldots, Y_r ) \mbox{ defines the directrix } 
	\Dir_M ( \IE ), 
	\eqno{(*)}
	\]
	where $ Y_j := y_j \mod M^2 $, for $ j \in \{ 1, \ldots, r \} $.
	Then there exist elements $ ( y^* ) = ( y^*_1, \ldots, y^*_r ) $ in $ R $ such that $ ( u, y^* ) $ is a {\RSP} for $ R $, condition $(*)$ holds for $ ( y^* ) $, and 
	$ 
		 \Delta ( \IE , u, y^* ) = \Delta ( \IE , u ).
	$ 
\end{Thm}

Let us point out that the situation is more general in \cite[Theorem~3.15]{BerndPoly}.
In particular, $ R $ is a regular local ring that is not necessarily complete. 
In order to get the elements $ (y^*) $ one has to pass to the $ M $-adic completion of $ R $.
It is a question for itself to consider whether it is possible to obtain such $ (y^*) $ without passing to the completion, see for example \cite{CPcompl,CSCcompl}.

The assumption $ (*)$ is crucial, see \cite[Example 3.12]{BerndPoly}.
Nonetheless, if $ (Y) $ has to be extended in order to achieve $ (*) $,
we can still recover information on the singularity, see Theorem~\ref{Thm:nuandsoon}\eqref{It:delta=1}.

	As we recalled in Example~\ref{Ex:PolyNotUnique}, $ \Delta (\IE ,u,y) $ (and thus $ \Delta (\IE, u) $) does not behave well under the equivalence relation $ \sim $.
	Nonetheless, the information that we extract from the polyhedron is an invariant of the idealistic exponent $ \IE_\sim $, see Theorem~\ref{Thm:nuandsoon}\eqref{It:delta_inv}.
	In \cite[p.~385]{BerndPoly}, we explain the concept of a characteristic polyhedron for an idealisitic exponent which is independent of the choice of a representative. 
	Since this is not needed here, we refer the reader for further information to \cite{BerndPoly}.

	We need the following theorem which follows from results in \cite{BerndPoly}.

\begin{Thm}
\label{Thm:nuandsoon}
	Let $ \IE = (J,b) $ 
	be a pair on $ R $ such that $ x \in \Sing ( \IE ) $.
	Let $ (u,y) = ( u_1,, \ldots, u_{ d };y_1, \ldots, y_s ) $ be a {\RSP} for $ R $.
\begin{enumerate}[(1)]
	\item\label{It:delta_ind_max} 
	Suppose $ R = K[[u,y]] $ and 
	that there exist differential operators $ \cD_{y,j} $ of order $ b-1 $ only involving $ (y)$ as well as elements $ f_j \in J $ such that $ \cD_{y,j} (f_j) = y_j $ for all $ j \in \{ 1, \ldots, s \} $.  
	Then $ V( y ) $ has maximal contact with $ \IE $ at $ M $ and the polyhedron 
	$ 
		\Delta (\IE, u,y)
	$
	is independent of the choice of $ ( y ) $ with theses properties,
	i.e.,
	if $ (z) $ is another extension of $ ( u ) $ to a {\RSP} for $ R $ which can be determined through differential operators of order $ b-1 $ in $ (z) $, then $\Delta (\IE, u,z) = \Delta (\IE, u,y) $.
	Moreover, if additionally condition $ (*) $ of Theorem~\ref{Thm:BerndHiro} holds,
	then
	$ 
	\Delta (\IE, u,y) = \Delta (\IE, u)
	$. 
	\item\label{It:delta_inv}
	The number
	$  \delta ( \IE, u )  :=  \delta ( \Delta (\IE ,u) ) 
	\in \frac{1}{b!} \,\IZ_{\geq 1} \cup \{\infty\}
	$ 
	does not depend on $ ( y ) $ and is invariant under the equivalence relation $ \sim $.
	Therefore $ \delta ( \IE; u ) $ is an invariant of the idealistic exponent $ \IE_\sim $ and $ ( u ) $.
	\item
	\label{It:delta=1}
	Suppose that $ (y, u_{ e + 1 }, \ldots, u_d ) $, for some $ e < d $ (and possibly renaming $ (u) $), fulfills condition $(*)$ of Theorem~\ref{Thm:BerndHiro}.
	Then we have
	$
		 \delta ( \IE;u_1,\ldots, u_d; y_1, \ldots, y_s ) = 1 
	$.
\end{enumerate}
\end{Thm} 

\begin{proof}
	The proof of Statement (1) follows the same lines as the one for \cite[Proposition~4.4]{BerndPoly}
	where the assumption on $ R $ is slightly different.
	This difference has no effect on the arguments in the proof.
	Statement (2) is shown in \cite[Theorem~4.8]{BerndPoly}
	and Statement (3) is proven in \cite[Lemma~4.9]{BerndPoly}.
\end{proof}

The following notion is the key ingredient to draw the connection to the invariant of Bierstone and Milman.

\begin{Def}
\label{Def:nu}
	Let $ \IE $ be a pair on $ R $ with $ x \in \Sing ( \IE ) $.
	Fix a system of elements $ ( u_1, \ldots, u_d ) $ in $ R $ which can be extended to a {\RSP} for $ R $ and let $ ( y ) = ( y_1, \ldots, y_s ) $ be such an extension of $ (u) $.
	\begin{enumerate}[(1)]
		\item\label{It:d_i}
		For $ i \in \{ 1, \ldots, d \} $, we define 
		\[ 
			d_i ( \IE, u , y  ) := \inf \{ v_i \mid v = ( v_1, \ldots, v_d ) \in \Delta (\IE, u,y) \} \in \frac{1}{b!} \, \IZ_{\geq 0}  \cup \{\infty \}.
		\]
		\item\label{It:nu_I} 
		Let $ I \subseteq \{ 1, \ldots, d \} $ be a subset.
		If $ \Delta (\IE, u,y) \neq \emptyset $, we define
		\[
			\nu_I ( \IE, u, y ) := \delta ( \IE, u, y ) - \sum_{i \in I } d_i ( \IE, u, y ) \in \frac{1}{b!}\, \IZ_{\geq 0} .
		\]
		If $ \Delta (\IE, u,y) = \emptyset $,
		we set $ \nu_I ( \IE, u, y ) := \infty $.
		\item
		Assume that condition $ (*) $ of Theorem~\ref{Thm:BerndHiro} holds for $ ( y ) $ and let $ ( y^* ) $ be a system of elements in $ R $ as in the conclusion of Theorem \ref{Thm:BerndHiro},
		i.e., 
		$  \Delta ( \IE , u , y^* )  = \Delta ( \IE , u ) $.
		For $ i \in \{ 1, \ldots, d \} $ resp.~$ I \subseteq \{ 1, \ldots, d \} $, we introduce
		\[
		d_i ( \IE, u  ) := d_i ( \IE, u ,y^*  )
		\hspace{10pt}\mbox{ and }\hspace{10pt}
		\nu_I ( \IE, u ) := \nu_I ( \IE, u, y^* ) 
		.
		\]
	\end{enumerate}
\end{Def}

It follows from Example \ref{Ex:PolyNotUnique} that $ \nu_I ( \IE, u ) $ and $ \nu_I ( \IE, u, y ) $ may change under the equivalence relation $ \sim $.
Therefore it is not an invariant of the idealistic exponent.
In order to take care of this we introduce the notion of pairs (and idealistic exponents) with history in Section~\ref{Sec:History}.
Since this phenomena only appears if there are exceptional divisors present, we first discuss the case without exceptional divisors in the following section.

\section{The case without exceptional divisors}
\label{Sec:NoExc}

We begin by fixing the setup and discussing the situation for the invariant of Bierstone and Milman if there are no exceptional divisors present. 
The general case with exceptional divisor will be considered in Section~\ref{Sec:BMfull}.

\begin{Setup}
	\label{SetupB}
	Let $ K $ be a field of characteristic zero and $ R = K[[T_1, \ldots, T_n]] $.
	Let $ J \subset R $ be a non-zero ideal
	and let $ (f) = (f_1, \ldots, f_m) $ be a standard basis for $ J $.
	The latter means $ f_1, \ldots, f_m $ are elements in $ J $ fulfilling the properties:
	\begin{enumerate}[(1)]
		\item $ \ini_M(J) = \langle \ini_M(f_1), \ldots, \ini_M(f_m) \rangle $, 
		\item $ \ini_M(f_i) \notin \langle \ini_M(f_1), \ldots, \ini_M(f_{i-1}) \rangle $ for $ i \in \{ 2, \ldots, m \} $, and
		\item $ b_1 \leq \cdots \leq b_m $,
		where $ b_i := \ord_M(f_i) $ for $ i \in \{ 1, \ldots, m \} $.
	\end{enumerate}
	Notice that a standard basis of $ J $ generates $ J $ by \cite[Corollary~(2.21.d)]{HiroCharPoly}.
	For this data, we introduce the pair 
	 $ 
	 			\IE := ( f_1, b_1 ) \cap \ldots \cap ( f_m, b_m ) .
	 $ 
	 
	 Let $ (u,y)  = (u_1, \ldots, u_e; y_1, \ldots, y_r ) $ be a {\RSP} for $ R $ such that the initial forms of $ (y) $ determine the directrix $ \Dir_M(\IE) $
	 and such that $ V(y) $ has maximal contact with $ \IE $. 
	 Additionally, suppose that $ (f) $ is normalized in the sense of \cite[Definition~1.13(2)]{CSCcompl}.
\end{Setup}

For our purpose it is not crucial to give the precise definition of being normalized.   
Therefore we skip this quite technical definition.
We only mention that the assumptions imply (using \cite[Theorem~(4.8)]{HiroCharPoly})
that $ \Delta (\IE,u,y) = \Delta(\IE,u) = \Delta(J,u) $,
where the latter is Hironaka's original characteristic polyhedron \cite{HiroCharPoly}.

In \cite[Lemma 3.1.5]{BerndThesis}, the author proved that the conditions of Setup \ref{SetupB} imply the original assumptions of Bierstone and Milman \cite[(7.2), p.~262.]{BMheavy}.
Since this part is not essential for the construction of the invariant we do not recall all the technical notation and refer the interested reader to \cite[Section 3.1]{BerndThesis}.

%
%
%
%
%
%
%
%
%
%
%
%
%
%

\begin{Constr}
	\label{Constr:NoExc}
	Let the situation be as in Setup \ref{SetupB}.
	Define
	\[ 
	\cG_1 := \IE =  ( f_1, b_1 ) \cap \ldots \cap  ( f_m, b_m ) .
	\]
	Choose $ i_0 \in \{ 1,\ldots, m\} $.
	Then $ \ord_{M} (f_{i_0}) = b_{i_0} $ and for simplicity, we write $ (f, b) $ instead of $ (f_{i_0}, b_{i_0}) $.
	After a linear coordinate change we may assume that $\dfrac{ \partial^{b} f }{ \partial y_{1}^{b} } \neq 0$.
	By applying the Weierstrass preparation theorem \cite[Theorem~9.2]{Lang}, 
	we may assume without loss of generality that $ f $ is of the form $ f = y_1^b + \sum\limits_{q=0}^{b-1} y_i^q \phi_q $. 
	Set 
	\[
	z_1 : = \frac{1}{b!} \cdot \frac{ \partial^{ b -1 } f }{ \partial y_{ 1 }^{ b-1 } } = y_1 + \frac{\phi_{b-1}}{b}
	\ \ \ 
	\mbox{ and }
	\ \ \
	N_1 := V( z_1 ) .
	\]
	By the Diff Theorem, Proposition \ref{Prop:basicandmore}\eqref{It:Diff}, we have
	$ 
	\cG_1  \sim \cG_1  \cap ( z_1, 1 )
	$ 
	and $ N_1  $ has maximal contact with $ \cG_1 $.
	Note that $ (u, z_1, y_2, \ldots, y_r ) $ is a {\RSP} for $ R $.
	In the next step we focus on the situation in $ N_1  $.
	Define the pair $ \cH_1  $ (on $ N_1 $) by
	\[
	\cH_1 := \bigcap_{j=1}^{m^{(2)}} (h_j, b_{h_j}) :=  \bigcap_{i = 1}^m \; \bigcap_{ q := q(i) = 0 }^{ b_i - 1 } \, \left( \, \frac{ \partial^q f_i }{ \partial y_1^q } \Big|_{ N_1 }, \, b_i - q \, \right) .
	\]
	This is the coefficient pair of $ \cG_1 $ \wrt $ ( z_1 ) $ (Definition \ref{Def:IdCoeffExp}).
	Define
	\[
	\mu_2 := 
	\ord_{M} ( \cH_1 ).
	\]
	For simplicity, we abuse notation here (and in the following)
	by writing $ \ord_M(.) $ even though, 
	to be precise, 
	it would be $ \ord_{M/\langle z_1 \rangle}(.) $. 
	Since there are no exceptional divisors present,
	 we define $ \nu_2 := \mu_2 $
	 and
	\[ 
	\cG_2 :=
	 \bigcap_{ j = 1 }^{m^{(2)}}  (g_j,c_j)  
	 :=	 \bigcap_{ j = 1 }^{m^{(2)}} \left( h_j, b_{h_j} \nu_2 \right) .
	\]
	This completes the first cycle of the process without exceptional divisors.
	We repeat with $ \cG_2   $ taking the role of $ \cG_1   $.
	By construction, there exists $ i_1 \in \{1, \ldots, m^{(2)} \} $ such that $ \ord_{ M } ( g_{i_1} ) = c_{i_1} $.
\end{Constr}

\begin{Obs}
	By construction, $ \cH_r  = \ID ( \IE, u, z ) $ and thus we have
	\[  
	\nu_{ r + 1} = \mu_{ r + 1}  = \delta ( \IE, u ) .
	\]
	
	Set $ (w_1, \ldots, w_n) = (u,y) $.
	Since $ \mu_{s + 1} = \ord_M (\cH_s  ) = \delta (\cG_1; w_1, \ldots, w_{ n -s } ; z_1, \ldots, z_s ) $ (with the abuse of notation mentioned in Construction~\ref{Constr:NoExc}), 
	we get (by Theorem~\ref{Thm:nuandsoon}\eqref{It:delta=1}) 
	\[  
		\mu_{s + 1 } = 1 = \delta ( \IE; w_1, \ldots, w_{ n- s } ) 
		\quad 
		\mbox{ if } s < r  \  (\mbox{or equivalently if }  d := n- s > e).
	\] 
	
	After $ r $ steps we start over with $ \cG_{r + 1 } $ instead of $ \cG_1 $.
	We determine the directrix $ \Dir_M ( \cG_{ r + 1 }  ) $, distinguish $ ( u ) = (u_1, \ldots, u_e ) $ into 
	$
	( u ) =  ( \,u_1,\, \ldots,\, u_{ e^{ (2) }  };\, y_{ r + 1 },\, \ldots,\, y_{ r^{ ( 2 ) } }\, ),
	$
	and so on.
\end{Obs}

\begin{Cor}
	\label{Cor:InvZero}
	Let the situation be as in Setup \ref{SetupB}.
	and use the notation of Construction~\ref{Constr:NoExc}.
	In the case without exceptional components the invariant of Bierstone and Milman has the following form
	\[ 
	\begin{array}{cl}
		\inv_X ( x ) 
		& 
		= ( \nu_1, s_1;\, \nu_2, s_2; \, \ldots) =
		\\[5pt]
		& 
		= ( \nu_1, 0;\, 1, 0;\, \ldots;\, 1, 0;\, \nu_{ r^{(1)} + 1 }, 0;\, 1, 0;\, \ldots;\, 1, 0;\, \nu_{ r^{ ( 2 ) } + 1 }, 0;\, 1, 0;\, \ldots),
	\end{array} 
	\]
	where $ r^{ ( 1 ) } := r $ and $ r^{ ( q ) } $, $ q \geq 3 $, is defined in the analogous way as $ r^{ ( 2 ) } $ above.
	Set $ r^{ ( 0 ) } := 0 $.
	Then we have $ 0 = r^{ ( 0 ) } < r^{ ( 1 ) } < r^{ ( 2 ) } < \cdots \leq n $ and, for all $ q \geq 1 $, $ (Y_{ r^{ ( q - 1 ) } + 1 }, \ldots, Y_{ r^{ ( q ) } } ) $ yields $ \Dir_M ( \cG_{ r^{ ( q - 1 ) } + 1 }) $ and
	introducing $ e^{ ( q ) } := n - r^{ ( q ) } $, we get:
	\[
	\nu_{ r^{ ( q ) } + 1 } \,=\, \delta ( \cG_{ r^{ ( q - 1 ) } + 1 } ; \, u_1, \ldots, u_{ e^{ (q ) } } )  \, > \, 1.
	\]
\end{Cor}

Recall that 
$
\delta (\cG_{ r^{ ( q - 1 )} + 1 }  ;  u_1, \ldots, u_{ e^{ (q ) } }  ) 
$
is an invariant of the idealistic exponent $ \cG_{ r^{ ( q - 1 )} + 1 }  $ and the system $ ( u_1, \ldots, u_{ e^{ (q ) } }) $ coming from the respective polyhedron (Theorem \ref{Thm:nuandsoon}\eqref{It:delta_inv}).
Hence we obtained Theorem \ref{MainThm:nuPurelyPoly} in the special case without exceptional divisors.
Furthermore, we have: 

\begin{Prop}
	\label{Prop:MainThmNoExc}
	Let the data be as in Setup \ref{SetupB} and use the notation of Construction~\ref{Constr:NoExc} and Corollary~\ref{Cor:InvZero}.
	Fix $ q \geq 1 $. 
	Let $ g_1, \ldots, g_\ell, d_1, \ldots, d_\ell $ be the elements obtained from $ \cG_1 $ following Construction~\ref{Constr:NoExc}  with the property
	\[
		\cG_{ r^{ ( q - 1 )} + 1  } = \bigcap_{i=1}^\ell \, (g_i, d_i)
		\quad \quad \mbox{(pair on $ V ( z_1, \ldots,  z_{r^{(q-1)}} ) $)}
	\]
	where $ d_i \leq \ord_M(g_i) $ for $ i \in \{ 1 , \ldots, \ell \} $.
	Then each $ g_i $  
	has an expansion of the form
	\begin{equation}
		\label{eq:ExpanFi}
		g_i = g_i ( u^{ (q) }, y^{(q)} ) = G_i( y^{(q)} )+ \sum_{|B| < d_i} G_{B,i}( u^{ (q) } )  \cdot  ( y^{(q)} )^B + g_i^\ast (u^{ (q) }, y^{ (q) } ),
	\end{equation}
	where $ ( u^{ ( q ) }, y^{ ( q ) } ) = ( u_1, \ldots, u_{ e^{ ( q ) } } ; y_{ r^{ ( q - 1 ) } + 1 }, \ldots, y_{ r^{ ( q ) } } ) $,
	we have  $ B = B(i) \in \IZ^{ r^{ ( q ) } - r^{ ( q - 1 ) } }_{\geq 0} $,
	and 
	$ g_i^\ast (u,y)  \in \langle y^{ ( q ) } \rangle^{ d_{ i } } \cap M^{ d_{ i } + 1  } $, 
	\begin{enumerate}[(1)]
		\item\label{It:NoExcHomog} 
		$ G_i ( y^{ ( q ) }  ) \in \resfield [ y^{ ( q ) } ] $ is a polynomial homogeneous of degree $ d_i $ if $ \ord_M(g_i) = d_i $
		and otherwise $ G_i ( y^{ ( q ) }  ) $ is the zero polynomial, 
		
		\item\label{It:NoExcHPolyPoart} 
		$ G_{B,i}(   u^{ ( q ) }   ) \in \resfield [[  u^{ ( q ) }  ]] $ with $ \ord_M ( \,G_{ B, i } \, ) > d_i -|B|$.
	\end{enumerate}
	Moreover, the following properties hold: 
	\begin{enumerate}[(1)]
		 \setcounter{enumi}{2}
		\item
		\label{It:NoExcNextStep}
		$ \displaystyle	\cH_{r^{ ( q ) } } = \bigcap_{i=1}^\ell \bigcap_{B  : |B| < d_i}  \left(\, G_{B,i}(    u^{ ( q ) }   ),\; d_i -|B| \,\right) $, 
		\\[5pt]
		$ \displaystyle \cG_{ r^{ ( q ) } + 1 } =  \bigcap_{i=1}^\ell \bigcap_{B  : |B| < d_i}  \left(\, G_{B,i}(   u^{ ( q ) }   ),\; ( d_i -|B|)\cdot \delta^{ ( q ) } \,\right) $, 
		\\[5pt]
		$ \nu_{ r^{ ( q ) }  + 1 } = \min \left\{\, \dfrac{\ord_M (\,G_{B,i} \,)}{ d_i -|B|} \,\;\Big|\,\; i, B : |B| < d_i \,\right\} = \delta^{ ( q ) }> 1 $,
		\\[5pt]
		where $ \delta^{ ( q ) } := \delta (  \cG_{ r^{ ( q - 1 ) } + 1  },\, u^{ ( q  ) }  ) =  \delta ( \Delta(\, \cG_{ r^{ ( q - 1 ) } + 1  } ,\, u^{ ( q  ) } \, ) ) $.
		\item
		\label{It:NoExcNextPoly}		
		The polyhedron $ \Delta ( \cG_{ r^{ ( q ) } + 1  }  , u^{ ( q + 1 ) } ) $ 
		is a projection of $ \Delta ( \cG_{ r^{ ( q - 1  ) } + 1  }  ,u^{ ( q ) } ) $. 
	\end{enumerate}

	On the other hand, let $ s \in \IZ_ + $ with $ r^{ ( q - 1 ) } < s <  r^{ ( q ) } $.
	Set $ (     u^{ ( q , s )  }    ) := ( u^{ ( q ) }, y_{ s + 1 }, \ldots, y_{ r^{ ( q ) } } ) $ and 
	$ (     y^{ ( q , s )  }    ) := ( y_{  r^{ ( q - 1 ) } + 1 }, \ldots, y_s ) $.
	The statements analogous to (\ref{eq:ExpanFi}) and \eqref{It:NoExcHomog}--\eqref{It:NoExcNextPoly} are true for $ (     u^{ ( q , s )  },  y^{ ( q , s )  }    ) $ instead of $ ( u^{ ( q ) }, y^{ ( q ) } ) $ 
	using  
	$ \delta^{ ( q , s )  } := \delta (  \Delta (  \cG_{ r^{ ( q - 1 ) } + 1  } ,\, u^{ ( q,s  )},\, y^{ ( q , s ) }  ) )  = 1 $.
\end{Prop}

\begin{proof}
	Assertions \eqref{It:NoExcHomog}, \eqref{It:NoExcHPolyPoart} and $ \delta^{ ( q ) } > 1  $  follow since $ ( Y^{ ( q ) } ) = (Y_{ r^{ ( q - 1 ) } + 1 }, \ldots, Y_{ r^{ ( q ) } } ) $ yields $ \Dir_M ( \cG_{ r^{ ( q - 1 ) } + 1 }  ) $.
	Part \eqref{It:NoExcNextStep} is a consequence of the definition of the coefficient pair (Definition \ref{Def:IdCoeffExp}) and the construction of  $	\cH_{r^{ ( q ) } }, \cG_{ r^{ ( q ) } + 1 }  $ and $ \nu_{ r^{ ( q ) }  + 1 } $ (Construction \ref{Constr:NoExc}).
	Since $ V( y^{ ( q ) } ) $ has maximal contact (which is determined through differential operators), the polyhedron 
	$ 
	\Delta ( \cG_{ r^{ ( q ) } + 1  }  ; u^{ ( q + 1 ) };  y_{ r^{ ( q - 1 )} + 1 },\, \ldots,\,  y_{ r^{ ( q  )} } ) $
	is minimal
	and hence equal to $  \Delta ( \cG_{ r^{ ( q ) } + 1  }  , u^{ ( q + 1 ) }  )  $ (Theorem~\ref{Thm:nuandsoon}\eqref{It:delta_ind_max}) which implies \eqref{It:NoExcNextPoly}.
	
	The proof of the last part for $ s \in \IZ_+ $ with $ r^{ ( q - 1 ) } < s <  r^{ ( q ) } $ is clear.
\end{proof}

\begin{Rk}
	A similar description of $ \cH_r  $ as above (with $ \nu_1 = \nu_1 ( x) = H_{X,x} $ the Hilbert-Samuel function of $ X = V(J) $) has already been proven in \cite[Construction 4.18 and Theorem 9.4]{BMheavy}.
	Further, they show how to get their invariant in the case without exceptional components by using ``weighted initial exponents'' and the ``weighted diagram of initial exponents'', see \cite[Remark~3.25]{BMheavy}.
\end{Rk}

So far we have to determine the elements $ ( g ) $ giving $ \cG_{ r^{ ( q - 1 )} + 1  } $ step-by-step and apply Proposition~\ref{Prop:MainThmNoExc}.
By introducing weights on the {\RSP} $ ( u, y) $ we are 
able to extend this result such that we get similar statements only with the use of the generators $ ( f ) $ of $ J $.
Polyhedra of pairs in the weighted-homogeneous setting are discussed in \cite[Remark~3.17]{BerndPoly}.

\begin{Obs}[A glimpse into the weighted setting]
	\label{Obs:quasihomogeneous}
	Let 
	$ \cG_1  = \bigcap_{i=1}^m ( f_i, b_i )  $ be as in Construction \ref{Constr:NoExc}.
	If we neglect exceptional components, we have
	\[ 
	\inv_X  (x)
	= ( \nu_1, 0;\, 1, 0;\, \ldots;\, 1, 0;\, \nu_{ r^{(1)} + 1 }, 0;\, 1, 0;\, \ldots;\, 1, 0;\, \nu_{ r^{ ( 2 ) } + 1 }, 0;\, 1, 0;\, \ldots).
	\]
	Set $ \delta^{(1)} :=  \nu_{ r^{(1)} + 1 } > 1  $.
	At the beginning we separated the {\RSP} of the regular local ring $ R $ into $ ( u, y ) $, where the initial forms of $ ( y ) = ( y_1, \ldots, y_r ) $ build a minimal generating set for the ideal of $ \Dir_M ( \cG_1  ) $.
	The latter is the directrix associated to the homogeneous ideal 
	$ 
	\dBu{ I }{ 0 } := \langle \ini_M ( f_i, b_i ) \mid  i \in \{ 1, \ldots, m \} \rangle = \ini_M(J) .
	$
	
	Up to now the images of $ ( u, y ) $ in the graded ring $ \gr_\maxIdeal ( R ) $ were equipped with the standard grading.
	In other words, we considered the weight $ \omega_0 $ on $ R $ given by
	$$
	\omega_0 ( y_j ) = \omega_0 (u_i) = 1
	\hspace{10pt} \mbox{ and } \hspace{10pt}
	\omega_0 ( \lambda ) = 0
	$$
	for $ j \in \{ 1, \ldots, r  \} $, $ i \in \{ 1, \ldots, e \} $ and $ \lambda \in R^\times $.
	Recall that $ \dBu{ r }{ 1 } = r $ and $ \dBu{ e }{ 1 } = e $.
	We now introduce the weight $ \omega_1 $ on $ R $ given as follows
	$$
	\omega_1 ( y_j ) = 1,
	\hspace{10pt}
	\omega_1 (u_i) = \dfrac{ 1 }{ \dBu{ \delta }{ 1 } }
	\hspace{10pt} \mbox{ and } \hspace{10pt}
	\omega_1 ( \lambda ) = 0
	$$
	for $ j \in \{ 1, \ldots, \dBu{ r }{ 1 }  \} $, $ i \in \{ 1, \ldots, \dBu{ e }{ 1 } \} $ and $ \lambda \in R^\times $.
	We have
	$  \omega_1 (u^A y^B) = \frac{|A|}{\delta^{(1)}} + |B| $.
	
	Observe that $ \omega_1 ( f_i ) = b_i $ for $ i \in \{ 1, \ldots, m \} $
	and we have $ \omega_1 (u^A y^B) = b_i $ if and only if 
	$ \frac{ | A | }{ b_i - | B | } =  \dBu{ \delta }{ 1 }  $
	if $ b_i - | B | \neq 0 $.
	Denote by $ \ini_{\omega_1} (f_i) $ the {\em $ \omega_1 $-initial form} of $ f_i $ which is the sum over those terms where $ \omega_1 $ attains its minimum.
	Notice that $ \ini_{\omega_1} (f_i) $ is equal to the $ \delta^{(1)} $-initial form of $ f_i $
	which is the sum over those terms contributing to the face of the Newton polyhedron determining $ \delta^{(1)} $.
	
	Consider the weighted-homogeneous ideal
	$
	\dBu{ I }{ 1 } :=  \langle \, \ini_{\omega_1} (f_i)  \mid i \in \{ 1, \ldots, m \} \, \rangle
	$.
	The notion of the directrix (Definition~\ref{Def:directrix_both}\eqref{It:DefDirI}) can be extended to weighted-homogeneous ideals. 
	Using our notion of the directrix of a pair (Definition~\ref{Def:directrix_both}\eqref{It:DefDirIE}), we can give an ad-hoc definition in characteristic zero:
	Choose $ (y) $ such that $ V(y) $ has maximal contact with $ \cG_1 $.
	Let $ \bigcap_{\iota \in \cI} (G_\iota, c_\iota) 
	:= \ID ( \bigcap_{i=1}^m (  \ini_{\omega_1} (f_i) ,  b_i  ) ;U;Y) $.
	Then we set
	\[
		\Dir_M(I^{(1)})
		:= 
		\Dir_M \Big(  
		(  \langle Y_1, \ldots, Y_r \rangle , \,  1 ) 
		\, \cap \, 
		\bigcap_{\iota \in \cI} \, (  G_\iota , \, c_\iota \cdot \delta^{(1)} )
		\Big).
	\]
	Using the notation of Proposition~\ref{Prop:MainThmNoExc},
	one can show that 
	$ \Dir_M(I^{(1)}) $ is determined by $ (y^{(12)} ) := (y^{(1)}, y^{(2)} ) = (y_1, \dots, y_{r^{(2)}}) $.
	Then, for each $ i \in \{ 1, \ldots, m \} $, 
	we can expand $ f_i $ \wrt $ (  \dBu{u}{2},  \dBu{y}{12} ) $ as
	$$
	f_i = f_i  (  \dBu{u}{2},  \dBu{y}{12} ) 
	= F_i( \dBu{y}{12} )
	+ \sum_{ \omega_1 ( {\dBu{y}{12}}^B ) < b_i} F_{B,i} ( \dBu{ u }{ 2 } ) \cdot {\dBu{y}{12}}^B + f_i^\ast ( \dBu{u}{2},  \dBu{y}{12}  ),
	$$
	where $  f_i^\ast ( \dBu{u}{2} ,  \dBu{y}{12}  ) \in \langle {\dBu{y}{12}}^B  \mid  \omega_1 ( {\dBu{y}{12}}^B ) \geq b_i  \rangle \cap \langle 
	g \in R \mid \omega_1 (g) \geq b_i+1 
	\rangle $,
	we have $ B = B(i)\in\IZ^{ r^{ ( 2 ) } }_{\geq 0} $
	and the following properties hold:
	\begin{enumerate}[(1)]
		\item  $ F_i ( \dBu{y}{12}  ) \in \resfield [ \dBu{y}{12} ] $ is a polynomial which is weighted-homogeneous of degree $ b_i $ (\wrt $ \omega_1 $).
		\item $ F_{B,i}(   \dBu{u}{2}  ) \in \resfield [[ \dBu{u}{2}   ]] $ and $ \omega_1( \,F_{ B, i } \, ) > b_i - \omega_1 ({\dBu{y}{12}}^B) $.
		\item 
		$	\displaystyle 
		\cH_{r^{ ( 2 ) } }  
		= \bigcap_{i=1}^m \bigcap_{B : \omega_1 ({\dBu{y}{12}}^B) < b_i}  
		\, \left(\, F_{B,i}(   \dBu{u}{2}  ) ,\; b_i - \omega_1 ({\dBu{y}{12}}^B) \,\right)  $,
		\\[5pt]
		$ \displaystyle 
		\cG_{ r^{ ( 2 ) } + 1 }  
		= \bigcap_{i=1}^m \bigcap_{B : \omega_1 ({\dBu{y}{12}}^B) < b_i}
		 \, \left(\, F_{B,i}(   \dBu{u}{2}  )   ,\; ( b_i - \omega_1 ({\dBu{y}{12}}^B))\cdot \delta^{ ( 2 ) } \,\right)  $, 
		\\[5pt]
		$ \nu_{ r^{ ( 2 ) }  + 1 } = \min \left\{\, \dfrac{ \omega_1 ( F_{ B, i } )}{ b_i - \omega_1 ({\dBu{y}{12}}^B)} \,\;\Big|\,\; i, B : \omega_1 ({\dBu{y}{12}}^B) < b_i \,\right\} = \delta^{ ( 2 ) }> 1 $,
		\\[5pt]
		where we further have
		$$ 
		\delta^{ ( 2 ) } = \delta (  \cG_{ 1  } , u^{ ( 2  ) } ,  \dBu{y}{12}  ) :=  \delta ( \Delta^{(1)} (\cG_{ 1  }  , u^{ ( 2  ) } ,  \dBu{y}{12}  ) ).
		$$
		Here $ \Delta^{(1)} ( \cG_{ 1  }  ,u^{ ( 2  ) } ,  \dBu{y}{12}  ) $ denotes the polyhedron in the weighted-homogeneous setting induced by $ \omega_1 $
		which is obtained by substituting 
		$ \dfrac{ A }{ b - |B|  } $ 
		by 
		\[
			\dfrac{ \ 
				\left( \dfrac{A_1}{\delta^{(1)}}, \ldots, \dfrac{A_{e^{(2)}}}{\delta^{(1)}}  \right) 
			 \ }{ b - \omega_1({\dBu{y}{12}}^B)  }.
		\]
		in the definition of the projected polyhedron (Definition~\ref{Def:PolyNonIntrinsic}).
	\end{enumerate}

	From here, we go on and repeat our considerations using the weight $ \omega_2 $ on $ R $ given by
	\[
	\begin{array}{ll}
		\omega_2 ( y_j ) = 1 ,
		& 
		\mbox{ if } j \in \{ 1, \ldots, \dBu{r}{1} \},
		\\[5pt]
		\omega_2 ( y_j ) = \dfrac{ 1 }{ \dBu{ \delta }{ 1 } } ,
		& 
		\mbox{ if } j \in \{ \dBu{r}{1} + 1 , \ldots, \dBu{r}{2} \} ,
		\\[5pt]
		\omega_2 (u_i) = \dfrac{ 1 }{ \dBu{ \delta }{ 1 } \, \dBu{ \delta }{ 2 }  }  ,
		& 
		\mbox{ for } i \in \{ 1, \ldots, \dBu{e}{2} \} ,
		\\[5pt]
		\omega_2 ( \lambda ) = 0 ,
		& 
		\mbox{ for } \lambda \in R^\times .
	\end{array} 	
	\]
	This will lead to a variant of Proposition \ref{Prop:MainThmNoExc} where we only use the generators $ ( f ) $ of $ J $.
\end{Obs}

In the upcoming work \cite{AQS} this weighted version of the characteristic polyhedron will play a crucial role for defining invariants of singularities over perfect fields of arbitrary characteristic.
Hence, further details on this perspective are postponed to \cite{AQS}.

A different weighted variant of the characteristic polyhedron appears in the characterization of quasi-ordinary hypersurface singularities by Mourtada and the author in \cite{MS}.

%
%
%
%
%
%
%
%
%
%
%
%
%
%

\section{Pairs and idealistic exponents with history}
\label{Sec:History}
In this section we introduce a refinement of pairs taking into account the exceptional data. 
We focus on the situation $ R = K[[T_1, \ldots, T_n]] $ 
and $ x \in \Spec (R) $ is the closed point corresponding to $ M = \langle T_1, \ldots, T_n \rangle $. 

\begin{Obs}
\label{Obs:OneBlowUpExcept}
	Let $ \IE = ( J, b) $ be a pair on $ R $ with $ x \in \Sing ( \IE ) $ and 
	let $ ( u, y ) = (u_1, \ldots, u_d; \, y_1, \ldots, y_s ) $ be a {\RSP} for $ R $ such that $ (y, u_{ e + 1 }, \ldots, u_d ) $, $  e \leq d $, defines the directrix $ \Dir_M ( \IE ) $.	
	\begin{enumerate}[(1)]
		\item
		Suppose that $ V (  y_1, \ldots, y_s ) $ has maximal contact with $ \IE $ at $ M $ (Definition~\ref{Def:maximal}).
		Let $ D = V(y,u_i \mid i \in I_D ) $, with $ I_D \subseteq \{ 1, \ldots, d \} $, be a permissible center for $ \IE $
		(Definition~\ref{Def:idexp_sing_perm}\eqref{It:permissible}).
		Set 
		\[  
			\delta_D  (\IE, u, y)  := 
			\min \left\{ \sum_{i \in I_D} v_i \mid v = ( v_1, \ldots, v_d ) \in \Delta (\IE, u, y) \right\}
			.
		\]
		Since $ D $ is permissible for $ \IE $, we have $ \delta_D  (\IE, u, y) \geq 1 $.
		
		Let $ \IE' $ be the transform of $ \IE $ under the blow-up with center $ D $ at a point $ x' $ lying above $ x $. 
		Let $ (u',y') $ be a regular system of parameters at $ x' $.
		Suppose the exceptional divisor is locally given by $ V ( u_j' ) = V(u_j) $.
		By using that $ V ( y ) $ has maximal contact, an easy computation shows that
		(for details see \cite[Observation~2.5.10]{BerndThesis})
		\[
			d_{ j } (\IE', u', y') = \delta_D  (\IE, u, y)  - 1,
		\]
		where $ d_{ j } (\IE', u', y') $ is introduced in Definition~\ref{Def:nu}\eqref{It:d_i}.
		\item
		During the resolution process we have to deal with the exceptional divisors $ E = \{ H_1, \ldots, H_\alpha \} $.
		We require on the resolution algorithm that the divisors associated to $ E $ have at most simple normal crossing singularities, i.e., each irreducible component is regular and they intersect transversally.
		Hence every $ H_i $ with $ x \in H_i $ is locally given by $ V ( T_i ) $, for suitably chosen $ (T_1, \ldots, T_n) $. 
	\end{enumerate}
\end{Obs}

The observation implies the following:
If  $ I \subseteq \{ 1, \ldots, d \} $ in Definition \ref{Def:nu} is the subset determined by the exceptional divisors containing the point $ x $, then the numbers $ d_i ( \IE, u, y ) $ are characterized by the preceding resolution process.
Therefore
\[ 
	\nu_I ( \IE, u, y ) = \delta ( \IE, u, y ) - \sum_{i \in I } d_i ( \IE, u, y ) \in \frac{1}{b!} \, \IZ^d_{\geq 0} 
\]
is an invariant depending on the preceding resolution process.

\begin{Def}
\label{Def:IEH}
	Let $ \IE = (J,b) $ be a pair on $ R $ with  $ x \in \Sing ( \IE ) $ 
	and let $ E = \{ H_1, \ldots, H_\alpha \} $, $ \alpha \in \IZ_+ $, be a set of irreducible divisors on $ \Spec(R) $ such that the associated divisor has at most simple normal crossing singularities.
	Let $ (u, y) = (u_1, \ldots, u_d; y_1, \ldots, y_s) $ be a regular system of parameters for $ R $ such that $ (y, u_{ e + 1 }, \ldots, u_d ) $, $  e \leq d $, defines $ \Dir_M ( \IE ) $. 
	\begin{enumerate}[(1)]
\item	
	We define the {\em exceptional data of $ \IE $ at $ x $ on $ V (y) $ associated to $ E $}
	as
	the map
	\[  
		\cE :=\cE_{ E, V(y) } (x) := \cE_E ( \IE, u, y ) \colon E \to \IR_{\geq 0}  
	\]
	where 
	\[ 
		\cE(H_i) 
		:= 
		d_i 
		:= 
		\begin{cases}
			d_i (\IE,u,y) & 
			\mbox{ if } x \in H_i \mbox{ and } V(y) \not\subset H_i, 
			\\
			0 
			& 
			\mbox{ if } x \notin H_i \mbox{ or } V(y) \subset H_i.  
		\end{cases}
	\] 
	This induces a map $ \cE_E  \colon 	\Sing ( \IE ) \to \{ \phi \colon  E \to \IR_{\geq 0} \} $.
	We call
	$
		(\IE,\cE_E) = ((J,b),\cE_E)
	$
	a \textit{pair with history} on $ R $ (resp.~on $ \Spec(R)$).
	If there is no confusion possible we skip the reference to $ E $ and write 
	$
	(\IE,\cE) .
	$ 

\item
	Let $ \IE_1 \sim \IE_2 $ be two {equivalent} pairs on $ R $ and $ E $ as above.
	The induced pairs with history $ (\IE_1,\cE_1) $ and $ (\IE_2,\cE_2) $ (with the same $ E $)
	are defined to be {\em equivalent at $ x \in \Sing ( \IE_1 ) = \Sing ( \IE_2 ) $ \wrt $ ( u, y ) $} if 
	$ \cE_1 (x) = \cE_2 (x) $.
	In particular, the assigned numbers $ d_j $ coincide.
	In this case we write $ \IE_1 \simHx \IE_2 $.
	
	We say $ (\IE_1,\cE_1) $ and $ (\IE_2,\cE_2) $ on $ R $ are \emph{equivalent} (with respect to $ (u,y) $), 
	if they are equivalent at any $ x \in \Sing ( \IE_1 ) = \Sing ( \IE_2 ) $ and we write $ \IE_1 \simH \IE_2 $.
	An \emph{\IEH} $ (\IE,\cE)_{\simH} $ denotes the equivalence class of a pair with history $ (\IE,\cE) $ \wrt the equivalence relation $ \simH $.
\end{enumerate}
\end{Def}

The definition of exceptional data presented here is slightly different but more precise than the one given in \cite{BerndThesis}.

Notice that $ \cE= \cE_{E,V(y)} (x) $ is the exceptional data on $ V(y) $, i.e., for $ H_i \cap V(y) $ with $ i \in \{ 1, \ldots, \alpha \} $. 
Therefore, any divisor $ H_i $ with $ V(y) \subset H_i $ gets mapped to zero along $ \cE $ since $ H_i \cap V(y)  = V(y) $ is not a divisor in $ V(y) $.

\begin{Def}
\label{Def:IEH_number}
	Let $ (\IE,\cE_E) $ be a pair with history on $ R $
	with $ E = \{ H_1, \ldots, H_\alpha \} $
	and $ (u,y) $ be as in Definition~\ref{Def:IEH}.
	We define
	\begin{equation}
		\label{eq:def_nu}  
		  \nu ( \IE, \cE_E, u, y ) :=
		  \left\{
		  \begin{array}{cl}
		  	\displaystyle 
		  	 \delta ( \IE, u ) - \sum_{ i = 1 }^\alpha d_i  \in \frac{1}{b!} \, \IZ_{\geq 0}, & \mbox{ if } e = d , \\ 
		  	 \displaystyle 
		   \delta ( \IE, u,y ) - \sum_{ i = 1 }^\alpha d_i  \in \frac{1}{b!} \, \IZ_{\geq 0}, & \mbox{ if } e < d.
		   \end{array}
		   \right.
	\end{equation}
	If $ \{ H_i \in E \mid x \in H_i \mbox{ and } V(y)\not\subset H_i \}  $ is locally at $ x $ defined by a subset $ I(\cE) \subset \{ 1, \ldots, d \} $ (i.e., $ (H_i)_x = V( u_{j(i)}) $, for some $ j(i) \in \{ 1, \ldots, d \} $, for those $ i $), then we have
	$
		 \nu ( \IE, \cE_E, u, y )  = \nu_{I(\cE)} ( \IE,u , y )
	$	
	(Definition~\ref{Def:nu}\eqref{It:nu_I}).
\end{Def}

Note that the distinction in \eqref{eq:def_nu} is required 
since $ \delta (\IE,u) $ is not defined if $ e < d $
(see the comment on assumption $ (\ast) $ after Theorem~\ref{Thm:BerndHiro}).
On the other hand, we know $ \delta ( \IE, u,y ) = 1 $ if $  e < d $ by Theorem~\ref{Thm:nuandsoon}\eqref{It:delta=1}.

\begin{Def}
	Let $ (\IE,\,\cE=\cE_E) $ (with $ E = \{ H_1, \ldots, H_\alpha \} $) be a {\IPH} on $ R $.
	A blow-up $ \pi \colon Z' \to Z $ with center $ D \subset Z $ is called {\em permissible for $ (\IE,\cE) $}, if the following conditions hold:
	\begin{enumerate}[(1)]
		\item	$ \pi $ is permissible for $ \IE $ (Definition~\ref{Def:idexp_sing_perm}\eqref{It:permissible}).
		\item	$ D \cup H_1 \cup \cdots \cup H_\alpha \subset Z $ have at most simple normal crossing singularities.
	\end{enumerate}
	The transform of $ (\IE, \cE_E) $ under a permissible blow-up $ \pi $ is given by 
$ ( \IE', \cE'_{E'} ) $, where $ \IE' $ denotes the transform of $ \IE $ under $ \pi $ 
and the exceptional data map $ \cE'_{E'} $ is defined by 
$ E' := \{ H'_1, \ldots, H'_\alpha, H_{\alpha + 1 } \} $.
	Here $ H'_i $ is the transform of $ H_i $ under the blow-up $ \pi $ 
and $ H_{ \alpha + 1 } $ is the exceptional divisor corresponding to $ \pi $.
\end{Def}

\begin{Rk}
	Reconsider Example~\ref{Ex:PolyNotUnique}.
	We have 
	$ \IE_1 = (y_1^d - u_1^{d-1} u_2^{d-1},\, d) \cap  (y_2,\, 1) $
	and 
	$ \IE_2 = (y_1^d - u_1^{d-1} u_2^{d-1},\, d) \cap (y_2^{d-1}-u_1^{d-2} u_2^{d-1},\, d-1 ) $.
	Suppose $ V ( u_1 ) $ is exceptional.
	We get 
	\[  
		d_1 (\IE_1,u,y) = \frac{d - 1}{d} 
		\quad \mbox{ and } \quad 
		d_1 (\IE_2,u,y) = \frac{ d - 2}{d - 1}.
	\]
	In \cite[Example~3.6]{BerndPoly},
	the equivalence $ \IE_1 \sim \IE_2 $ is deduced by applying the Diff Theorem (Proposition~\ref{Prop:basicandmore}\eqref{It:Diff}).
	Since $ d_1 (\IE_1,u,y) \neq d_1 (\IE_2,u,y ) $,
	it follows that $ \IE_1 $ and $ \IE_2 $ are not equivalent as {\IPsH}, because they have different exceptional data.
	
	Therefore the Diff Theorem is in general not true for the equivalence relation $ \simHx $.
	Let us explain this a bit more in details in a simplified variant of the example: 
	Suppose the Diff Theorem would be true for $ \simHx $.
	In particular, by applying the derivative $ \frac{\partial}{\partial u_1} $,
	we would get the equivalence
	\[
		\widetilde{\IE}_1:=
		(y_1^d - u_1^{d-1} u_2^{d-1},\, d) 
		\, \simHx \,  
		(y_1^d - u_1^{d-1} u_2^{d-1},\, d) \, \cap \, (u_1^{d-2} u_2^{d-1},\, d-1)
		=: \widetilde{\IE}_2
		.
	\]
	We have 
	$ d_1 (\widetilde{\IE}_1,u,y) = \frac{d - 1}{d} \neq \frac{d - 2}{d-1} = 
	d_1 (\widetilde{\IE}_2,u,y) $ which is a contradiction to 
	$ \widetilde{\IE}_1 \simHx \widetilde{\IE}_2 $ if $ u_1 $ is exceptional.
	  
	We provide a weaker version of the Diff Theorem which is valid for {\IEsH} in Lemma~\ref{Lem:BasicHist}\eqref{It:Diff_history}.
\end{Rk}

\begin{Prop} 
	\label{Prop:nuinvariant}
	Let $ (\IE,\cE_E) $ be a {\IPH} on $ R $ with $ E = \{ H_1, \ldots, H_\alpha \} $.
	Let
	$ x \in \Sing ( \IE ) $ and $ (u, y) = (u_1, \ldots, u_d; y_1, \ldots, y_s) $ a {\RSP} for $ R $ such that $ (y, u_{ e + 1 }, \ldots, u_d ) $, $  e \leq d $, defines the directrix $ \Dir_M ( \IE ) $. 
	
	Then $ \nu ( \IE, \cE_E, u, y ) $ is independent of $ ( y ) $ and invariant under $ \simHx $.
	Therefore it is justified to write
	$
		\nu ( \IE, \cE_E, u ) := \nu ( \IE, \cE_E, u, y)
	$
	and this is an invariant of the {\IEH}. 
\end{Prop}

\begin{proof}
	By Theorem~\ref{Thm:nuandsoon}\eqref{It:delta_ind_max} and \eqref{It:delta_inv}, $ \delta ( \IE, u, y ) $ is independent of $ (y) $ and invariant under $ \sim $.
	Hence it is also invariant under $ \simHx $.
	The exceptional data $ \cE (H_i) =  d_i $, $ i \in \{ 1, \ldots, \alpha \} $, is fixed under $ \simHx $. 
	By the definition of $  \nu ( \IE, \cE_E, u, y )$ this implies the assertion.
\end{proof}

The following properties hold for the refined equivalence $ \simHx $.

\begin{Lem}
\label{Lem:BasicHist}
	Let $ (\IE, \cE_E) = ((J,b), \cE ) $, $ x \in \Sing ( \IE ) $ and $ (u, y) $
	be as in Proposition~\ref{Prop:nuinvariant}.
	\begin{enumerate}[(1)]
		\item\label{It:1}		
		For every $ a \in \IZ_+$ we have $ (J^a, a b) \simHx (J,b) $.
		\item\label{It:2} 	
		Suppose there is another choice for $ ( y ) $, say $ ( z ) = (z_1, \ldots, z_s ) $,
						such that we have $ (z,1) \cap \IE \simHx (y,1) \cap \IE $.
						Then 
						$ 
							\ID (\IE , u, z) \simHx \ID (\IE , u, y)  
						$ 
						(both times with the induced exceptional data on $ V(z) $ and $ V(y) $).
		\item\label{It:3}	
		If $ char (\resfield) = 0 $ or if $ b < char (\resfield) $, then there exists a choice for the system $ (y)= ( y_1, \ldots, y_s ) $ such that 
						$ 
							\IE \simHx (y,1) \cap  \ID (\IE, u, y ) 
						$ 
						(with the induced exceptional data on $ V(y) $).
		\item\label{It:Diff_history}\label{It:4}	
		Let $ \cD_{ Q, u }^{ \log } = u^Q \cD_{ Q, u } \in \Diff^{\leq q}_{\resfield} \left( R \right) $, $ Q = (Q_1, \ldots, Q_d ) \in \IZ^d_{\geq 0} $ with $ |Q| = q $, be the logarithmic differential operators given by 
						$ 
							\cD_{ Q, u }^{\log} \left( u^A y^B \right) = \binom{A}{Q} u^{A} y^B.
						$
						Then we have 
						\[ 
							(J, b) \cap ( \cD_{ Q, u }^{\log} (J), b - q) \simHx (J, b) 
						\]
						(with the induced exceptional data on $ V(y) $).
						Moreover, if $ Q_i = 0 $ for all $ i \in \{ 1, \ldots, d \} $ with $ d_i \neq 0 $ in $ \cE ( x ) $, then the analogous statement is true for $  \cD_{ Q, u } $.
	\end{enumerate}
						Let $ \IE_1 = (J_1, b_1) $, $ \IE_2 = ( J_2, b_2 ) $ be two pairs on $ R $ such that both are equipped with the same exceptional data $ \cE_1 ( x ) = \cE_2 ( x ) = \cE ( x)  $ on $ V( y ) $ at $ x $.
						Suppose $ x \in \Sing ( \IE_1 ) \cap \Sing ( \IE_2 ) $.
	\begin{enumerate}[(1)]
		\setcounter{enumi}{4}
		\item
		\label{It:5}	Assume $ b_1, b_2 \in \IZ_+ $ and let $ m \in \IZ_+ $ be a positive integer such that $ b_1 \mid m $ and $ b_2 \mid m $.
						Then  $	(J_1,b_1) \cap (J_2,b_2) \simHx \left(J_1^{\frac{m}{b_1}} + J_2^{\frac{m}{b_2}}, m\right) $.
		\item\label{It:6}		$ \IE_1 \simHx \IE_2 $ implies
		\begin{enumerate}[(a)]
			\item\label{It:6a}	$ \IE_1 \cap \IE \simHx \IE_2 \cap \IE $.
			\item\label{It:6b}		$ \ord_P (\IE_1) = \ord_P (\IE_2) $, for all $ P \in \Spec(R) $.
							In particular, $ \ord_M (\IE_1) = \ord_M (\IE_2) $.
			\item\label{It:6c}	$ \Sing(\IE_1) = \Sing (\IE_2) $.
			\item\label{It:6d}		$ \ITC_M (\IE_1) \simHx  \ITC_M (\IE_2) $ and $ \IDir_M (\IE_1) \simHx \IDir_M (\IE_2) $.
			\item\label{It:6e}	$ \ID (\IE_1 , u, y) \simHx \ID (\IE_2 , u, y)  $ and $ (y,1) \cap \IE \simHx (y,1) \cap \ID (\IE , u, y) $ (both with the induced exceptional data on $ V(y) $).
		\end{enumerate}	
	\end{enumerate}
\end{Lem}

\begin{proof}
	The statements follow from a study of the behavior of the exceptional data
	and the corresponding statements before.
	More precisely, 
	\eqref{It:1} uses Proposition~\ref{Prop:basicandmore}\eqref{It:Power},
	\eqref{It:2} uses Proposition~\ref{Prop:moreCoef}\eqref{It:291b},
	\eqref{It:3} uses  Proposition~\ref{Prop:moreCoef}\eqref{It:sim_intersec},
	\eqref{It:4} uses  Proposition~\ref{Prop:basicandmore}\eqref{It:Diff},
	\eqref{It:5} uses Proposition~\ref{Prop:basicandmore}\eqref{It:intersec},
	\eqref{It:6a} uses Proposition~\ref{Prop:basicandmore}\eqref{It:sim_intersec},
	\eqref{It:6b} and
	\eqref{It:6c} use Proposition~\ref{Prop:basicandmore}\eqref{It:Numerical},
	\eqref{It:6d} uses Proposition~\ref{Prop:Dir}\eqref{It:Dir2},
	and
	\eqref{It:6e} uses Proposition~\ref{Prop:moreCoef}\eqref{It:291a} and \eqref{It:sim_DE}.
	For more details see \cite[Lemma~2.6.7]{BerndThesis}.
\end{proof}

\begin{Lem}
	Let $ \IE_1 = (J_1, b_1) $, $ \IE_2 = ( J_2, b_2 ) $ be two pairs on $ R $, $ x \in \Sing ( \IE_1 ) \cap \Sing ( \IE_2 ) $ and $ (u, y) $ a {\RSP} for $  R $.
	Let $ E = \{ H_1, \ldots, H_\alpha \} $ be a set of irreducible divisors on $ \Spec(R) $ such that the associated divisor has at most simple normal crossing singularities.
	Suppose both pairs have the same exceptional data 
	on $ V(y) $,
	$
		 \cE ( H_i ) = d_i 
	$
	for $ i \in \{ 1, \ldots, \alpha \} $. 
	Let $ \pi \colon Z' \to Z $ be a blow-up which is permissible for both pairs and let $ x' \in \Sing ( \IE'_1 ) \cap \Sing ( \IE'_2 ) $ with $ \pi ( x') = x $.
	We have 
	\begin{enumerate}[(1)]
		\item	$ (\IE_1 \cap \IE_ 2)' \simHxpi \IE'_1 \cap \IE'_2 $.
		\item	$ \IE_1 \simHx \IE_2 $ implies $ \IE_1' \simHxpi \IE_2' $.
	\end{enumerate}
\end{Lem} 

\begin{proof}
	If $ x $ is not contained in the center of the blow-up $ \pi $, then the situation at $ x' $ did not change.
	In this case the lemma is true.
	Thus let us assume that the center contains $ x $.
	Since the exceptional data are equal for $ \IE_1 $ and $ \IE_2 $ 
	the exceptional data at $ x' $ for $ \IE_1' $, $ \IE'_2 $, $ (\IE_1 \cap \IE_ 2)' $ and $ \IE'_1 \cap \IE'_2 $ are  given by the transform of $ \cE ( x ) $.
	Hence the first part follows by Proposition~\ref{Prop:basicandmore}\eqref{It:intersec_bu}.
	The definition of $ \sim $ implies the second part.	
\end{proof}

%
%
%
%
%
%
%
%
%
%
%
%
%
%

\section{The general case}
\label{Sec:BMfull}

After introducing pairs with history we are ready to turn to the general case involving exceptional divisors. 
First, we recall the construction of the invariant of Bierstone and Milman in the general case.
For this, we take a global perspective. 
Let $ X_0 $ be a scheme embedded in some regular scheme $ Z_0 $ of finite type over a field $ K $ of characteristic zero.
Suppose that we are in the $ j $-th step of a desingularization process for $ X_0 $. 
Hence, we have a sequence of the following form:
\begin{equation}
\label{eq:bigseq}
\begin{array}{rcccccccccccc}
		\emptyset=E_0	&	& E_1	&	&	\cdots &		E_{i} &	&	\cdots &	&	E_{j-1}	& &	E_j \\[8pt]
		Z_0					& \stackrel{\pi_1}{\longleftarrow} &	Z_1	& \stackrel{\pi_2}{\longleftarrow} &	\cdots  &	Z_{i} & \stackrel{\pi_{i + 1 }}{\longleftarrow}  & \cdots 	& \stackrel{\pi_{j-1}}{\longleftarrow} &	Z_{j-1}		& \stackrel{\pi_j}{\longleftarrow} 	&	Z_j	\\[5pt]
				\bigcup \, &	& \bigcup	&	& &		\bigcup &	& &	&	\bigcup	& &	\bigcup\\[5pt]
		X_0					& \longleftarrow &	X_1	& \longleftarrow &	\cdots	& 	X_{i}	& \longleftarrow &	\cdots	& \longleftarrow &	X_{j-1}		& \longleftarrow 	&	X_j	\\
				& &	&  & &   \vin &  &	&  &	 \vin & 	& \vin\\
				& &	&  & &   x_i & \leftmapsto & \cdots	& \leftmapsto &	 x_{j-1} & \leftmapsto &  x_j 
\end{array}
\end{equation}
\noindent
where each $ \pi_{i+1} \colon Z_{i+1} \rightarrow Z_{i }$ is a blow-up in a regular center which is contained in the singular locus of $ X_i $ and has at most simple normal crossings with $ E_i $, $ E_i $ denotes the set of exceptional divisors on $ Z_i $ corresponding to the former blow-ups and $ X_i $ is the transform of $ X_0 $ in $ Z_i $.
(The last row in \eqref{eq:bigseq} of points in $ X_i, \ldots, X_j $ is needed later).

Let $ x \in X := X_j \subset Z := Z_j $.
The invariant has the form 
$ 
	\inv_X (x) = (\nu_1, s_1;\, \nu_2, s_2; \, \ldots) .
$ 
Recall that $ \nu_1 = \nu_1 (x) = H_{X,x} $ is the Hilbert-Samuel function of $ X $ at $ x $.
Set 
\[
	\inv_r (x) := (\nu_1, s_1; \, \ldots; \, \nu_r, s_r) 
		\quad \mbox{ and } \quad 
	\inv_{ r + \frac{1}{2} } (x) := (\nu_1, s_1; \, \ldots; \, \nu_r, s_r;\, \nu_{r+1})
\]
for the truncation of the invariant after $ s_r $, resp. after $ \nu_{r+1} $.
We denote by $ E_j (x) $ the set of exceptional components passing through $ x $.

\begin{Constr}[$ \boldsymbol{ s_r ( x ) } $]
	\label{Constr:s_i}
	For $ i \leq j $ we denote by $ \pi_{ij} \colon Z_j \rightarrow Z_i $ the composition $ \pi_{ij} := \pi_{i + 1} \circ \pi_{ i +2 } \circ \cdots \circ \pi_{j-1} \circ \pi_j $ (with $ \pi_{jj} := \operatorname{id}_{Z_j} $) 
	and $ x_i:=\pi_{ij} (x) $ is the image of $ x = x_j \in Z_j$ in $ Z_i $, cf.~\eqref{eq:bigseq}.
	Let
	\[
		i_1 := \min \left\{ \; k \in \{0,\ldots, j\}  \,\mid\, \inv_{ \frac{1}{2} } (x) =\inv_{ \frac{1}{2} } (x_k) \; \right\}.
	\]
	We define $ E^1 (x) \subseteq E_j (x) $ to be the set of those exceptional components which are the strict transform of an exceptional component in $ E_{i_1} ( x_{i_1} ) $, 
	\[
		E^1 ( x) = \left\{ \; H \in E_j( x) \,\mid\, \exists \; H_0 \in E_{i_1} (x_{i_1})\,:\, H \mbox{ is the strict transform of } H_0 \;\right\}.
	\]
	Using these notions, define 
	$ 
		s_1( x ) := \# E^1 (x)
	$ and $ 
		\cE_1 (x) := E_j (x) \setminus E^1 (x).
	$ 
	
	Suppose we know $ \inv_{ r+\frac{1}{2} }(x) = (\nu_1, s_1; \, \ldots; \, \nu_r, s_r;\, \nu_{ r + 1 } ) $, for some $ r \geq 1 $. 
	Then we also constructed 
	$ \cE_r (x) = \cE_{r-1}(x) \setminus E^r(x)=E_j(x) \setminus \bigcup_{k=1}^r E^k(x) $.
	Let
	\[ 
		i_{ r + 1 } := \min \left\{ \; k \in \{0,\ldots, j\}  \,\mid\, \inv_{ r + \frac{1}{2} } (x) = \inv_{ r + \frac{1}{2} }(x_k) \; \right\} \geq i_r.
	\]
	Define $ E^{ r + 1 }(x) \subseteq \cE_r (x) $ to be the set of those exceptional components coming from the $ i_{ r + 1 } $-th resolution step and which we did not yet take into account in $ E^1(x), \ldots, E^r (x)$,
	\[
		E^{r+1}(x) = \left\{ \; H\in \cE_r(x)  \,\mid\, \exists \; H_0 \in E_{i_{r+1}} (x_{i_{r+1}})\,:\, H \mbox{ strict transform of } H_0 \;\right\}.
	\]
	Then we set 
	$ 
		s_{r+1} (x) := \#E^{ r + 1 }(x)
	$ and $ 
		\cE_{r+1} (x) := \cE_{r}(x) \setminus E^{r+1}(x). 
	$ 
\end{Constr}

For the definition of the terms $ \nu_{r+1} = \nu_{r+1} ( x ) $, $ r \geq 1 $, it is important to keep track of the ambient scheme and the  exceptional components.
In \cite{BMheavy} this is done by considering triples 
$
(N_{r},\cG_{r+1} , \cE_{r}) = (N_{r}(x),\cG_{r+1}(x) , \cE_{r}(x))
$,
where $ N_{r} $ is a regular ambient scheme contained in $ \Spec(R) $, 
$ \cG_{r+1} $ is a local description of the singularity $ x \in X $ on $ N_{r } $ and 
$ \cE_{r} $ is an ordered set of exceptional divisors on $ \Spec(R) $ which have simultaneously only normal crossing with $ N_{r} $.
In our language this corresponds to $ ( \cG_{r+1} , \cE_{r}  )  $ is a {\IPH} on $ N_{r} $ (Definition \ref{Def:IEH}) where we identify $ \cE_{r} $ with the exceptional data which it defines on $ N_{r} $.

\begin{Constr}[$ \boldsymbol{\nu_{r+1}(x)} $]
\label{Constr:nu}
	Let $ \widehat{\cO}_{X,x} \cong R/J $
	with $ R = K[[T_1, \ldots, T_n]] $ and $ J \subset R $ a non-zero ideal.  	
	Let $ ( f ) = ( f_1, \ldots, f_m ) $ be a (normalized) standard basis for $ J  $ as in Setup~\ref{SetupB}
	and set $ b_i := \ord_M(f_i) $ for $ i \in \{ 1, \ldots, m \} $.
	Consider $ \cG  =  ( f_1, b_1 ) \cap \ldots  \cap ( f_m, b_m ) $ on $ R $.
	 
	We start with the {\IPH} 
	$ 
		( \cG_1,\cE_0) := (\cG  , E_j ) 
	$ on $ N_0 := \Spec(R) $.
	We determine $ E^1  = E^1( x )$ and $ \cE_1 = \cE_1( x ) $ as described in Construction~\ref{Constr:s_i} and set
	\[
		\cF_1 :=  \cG_1 \cap \left(E^1 ,1 \right)
	\]
	where $ \left( E^1 , 1 \right) := \bigcap_{ H \in E^1  } \, (T_H, 1) $ and $ T_H $ denotes a local generator of $ H $.
	Thus we get the {\IPH} 
	$ 
		( \cF_1,\cE_1 ) 
	$
	on	$ N_0 $.
		
	Notice that also the exceptional data has changed.
	There are maybe less components
	and the assigned numbers may differ from those of the previous exceptional data.
	(For example, if $ E^1  \neq \emptyset $, then all the assigned numbers in $ \cE_1  $ are zero, because $ E_j $ defines a simple normal crossing divisor).
	
	As in Construction~\ref{Constr:NoExc}, we choose a maximal contact hypersurface $ N_1 := V ( y_1 ) $ for the pair $ ( \cF_1,\cE_1 ) $.
	Without loss of generality, we suppose that  $ (T_1, \ldots, T_{n-1}, y_1 ) $ is a regular system of parameters for $ R $.
	Let 
	\[ 
		\cH_1  := \ID ( \cF_1 ; T_1, \ldots, T_{n-1} ; y_1 )
	\] 
	be the coefficient pair of $ \cF_1 $ \wrt $ N_1 $.
	Hence, we obtained the {\IPH}
	$ 
		( \cH_1,\cE_1 ) 
	$ 
	on 
	$ N_1 $. 
	Again the exceptional data possibly has changed, because we have to consider here $ \cE_1 $ as data on $ N_1 = V ( y_1 ) $. 
	
	Recall, if we write $ \cF_1 = ( f_1, b_1 ) \cap \ldots \cap ( f_k, b_k ) $ ($ k \in \IZ_ + $, $ k \geq m $), then  		
		\[
			\cH_1  =
							\bigcap_{ i = 1 }^k
							\bigcap_{ q := q(i) = 0 }^{ b_i - 1}		
					 \,
			 	 	\left(\, \frac{ \partial^q f_i }{ \partial y_1^q } \Big|_{ N_1}, \, b_i - q \,\right) 
			 	 	.
		\]
	We put $ h_{i,q} := \dfrac{ \partial^q f_i }{ \partial y_1^q } \Big|_{ N_1 } $ for all $ i, q $ as above.
	Using this notation, one defines
	(always with $ i \in \{ 1, \ldots, k \} $ and $ q := q(i) \in \{ 0, \ldots, b_i - 1 \} $)
	\begin{equation}
		\left\{\hspace{15pt}
		\begin{array}{l}
			\mu_2 := \mu_2 ( x ) :=\min \left\{ \;\dfrac{ \ord_M ( h_{i,q} ) }{ b_i - q } \,\; \Big| \;\, i,\, q\, \right\} ,
				\\[15pt]
			\mu_{2,H} := \mu_{2,H} ( x ) := \min \left\{ \;\dfrac{ \ord_{H} ( h_{i,q} ) }{ b_i - q } \,\; \Big| \; \, i,\, q\, \right\} , \hspace{10pt} \mbox{ for } H \in \cE_1 =\cE_1(  x )  ,
				\\[15pt]
			\displaystyle 
			\nu_2 := \nu_2 ( x ) :=  \mu_2 - \sum\limits_{ H \in \cE_1 } \mu_{2,H},
		\end{array}
		\hspace{15pt}\right.
	\end{equation}
	where $ \ord_{H} ( h_{i,q} ) $ denotes the order of $ h_{i,q} $ along $ H $, i.e., if $ T_H $ is a local generator of $ H \in \cE_1 $, then 
	$ 
		\ord_{H} ( h_{i,q} ) = \max \left\{ \kappa \in \IZ_{\geq 0} \cup \{\infty\} \;\big|\; T_H^\kappa \mbox{ divides } h_{i,q} \right\}
		$.
	Furthermore, we still apply the abuse of notation introduced in Construction~\ref{Constr:NoExc} when we write $ \ord_M (.) $ for $ \ord_{M/\langle y_1 \rangle} (.) $.

	If $ \nu_2( x ) \in \{ 0, \infty\} $, then the process ends and the invariant is defined as 
	$$ 
		\inv_X (x) := \inv_{ 1\frac{1}{2} } ( x ) = (\nu_1, s_1;\, \nu_2).
	$$
	
	Suppose $ 0 < \nu_2( x ) < 1 $.
	We consider
	$$
		D_2 := D_2( x ) := \prod_{ H \in \cE_1 } T_H^{\mu_{2,H}  },
	$$
	where $ T_H $ denotes a local generator of $ H \in \cE_1 $.
	(We allow here fractional exponents).
	Then by definition of the terms $ \mu_{2,H} (x) $, each $ h_{i,q} $ 
	can be written as 
	\[
		h_{i,q} = D_2^{b_i - q } \cdot g_{i,q}  \, ,
	\]
	for some element $ g_{i,q} $. 
	We define the new pair
	\begin{equation}
	\label{Def:cG2(x)}
		\cG_2 := \cG_2 ( x ) :=  \Bigg( \, \bigcap_{  i = 1  }^k \;\, \bigcap_{ q = q(i) = 0 }^{ b_i - 1 } (\, g_{i,q},\, ( b_i - q ) \cdot \nu_2\,) \,\Bigg) \cap (D_2, 1 - \nu_2)
		\hspace{10pt}
		\mbox{ on }
		N_1.
	\end{equation}
	This is our variant of the so called companion ideal, thus we call it the {\em companion pair}.
	Clearly, the exceptional data has changed again.
	
	If $1 \leq \nu_2 (x) < \infty $, the assigned number $ 1 - \nu_2 $ of the $ D_2  $-component is $ \leq 0 $ and hence can be omitted, $ \displaystyle \cG_2 :=\cG_2( x ) :=  \bigcap_{  i = 1  }^k \;\, \bigcap_{ q = 0 }^{ b_i - 1 } ( \,g_{i,q},\, ( b_i - q ) \cdot \nu_2\,)  $.
	
	Together we get for $ 0 < \nu_2 ( x ) < \infty $ the {\IPH}
	$ (\cG_2, \cE_1 ) $ on $ N_1 $.
	This completes the first cycle of the general procedure.
	Then we start again with  $ ( \cG_2, \cE_1) $ on $ N_1 $ taking the role of the {\IPH} $ ( \cG_1, \cE_0) $ on $ N_0 $.

	\hfill {\em (End of Construction~\ref{Constr:nu})}
\end{Constr}

\begin{Obs}
	In Construction~\ref{Constr:nu},  $ \mu_{2,H}  $ coincides with the assigned number of $ H $ in the exceptional data $ \cE_1 = \cE_{1,V(y_1)} $ on $ V(y_1) $ of the {\IPH} $ (\cF_1 , \cE_1 ) $. 	
	Further, we have 
	$ 
	\Delta^N (\cH_1 ; T_1, \ldots, T_{ n-1 } )  
	= \Delta ( \cF_1 ;T_1, \ldots, T_{ n-1 } ; y_1 )
	$ 
	and $ \mu_2 = \delta ( \Delta ( \cF_1 ;T_1, \ldots, T_{ n-1 } ; y_1 ) ) $.
\end{Obs}

\begin{Lem}
\label{Lem:ObjectsEquiv}
	Let $ \cG_1  $ and $ \cG'_1  $ be two equivalent {\IPH}.
	Then
	$$
		\cF_1  \simHx \cF'_1 ,
		\hspace{15pt}
		\cH_1  \simHx \cH'_1  
		\hspace{15pt}
		\mbox{ and }
		\hspace{15pt}
		\cG_2 \simHx \cG'_2 ,
	$$
	where we have to consider the induced exceptional data.
	Thus these objects are invariants of the {\IEH} corresponding to $ \cG_1  $.
\end{Lem}

\begin{proof}
	The first and the second equivalence follow by Lemma \ref{Lem:BasicHist} \eqref{It:6a} and \eqref{It:6e}, respectively.
	The last equivalence $ \cG_2  \simHx \cG'_2 $ is clear for the case
		$ \cH_1 = (J, b ) $ and $ \cH'_1 = ( J^a, a b ) $ for some $ a \in \IZ_+ $.
	Thus we may assume $ \cH_1  = ( J, b ) $ and $ \cH'_1  = ( J', b ) $ 
	with the same assigned number $ b \in \IZ_+ $.
	For an element $ h \in J $ we have defined $ g = g ( h ) $ via $ h = D_2^b \cdot g $.
	Set $ I := \langle \, g ( h ) \mid h \in J \, \rangle $, then $ J = D_2^b \cdot I $.
	(Here we identify $ D_2^b $ with the ideal which it generates in $ R $).
	We have $ \cG_2  = ( I, \nu_2 b ) \cap ( D_2, 1 - \nu_2 ) $.
	We can do the same for $ \cH_1'  $ and obtain the ideal $ I' $ with the analogous property.
	
	If we can show $  ( I, \nu_2 b ) \simHx  ( I', \nu_2 b ) $ (as {\IPsH} on $ R $), then the assertion follows.
	Since we have factored $ D_2 $, the assigned numbers in the induced exceptional data on $ V(y_1) $ are all zero.
	Thus we only have to prove 
	$  
		( I, \nu_2 b ) \sim  ( I', \nu_2 b ) .
	$
	
	An extension of the {\RSP} by further independent elements does not change the situation.
	Hence we may assume that the extension is trivial.
	Further, we have, for any prime ideal $ P \subset  R  $,
	\[ 
		\ord_{ P } ( I ) = \ord_{ P } ( J ) - \ord_{P} ( D_2^b ) 
	= \ord_{P} ( J' ) - \ord_{P} ( D_2^b ) = \ord_{ P } ( I' ).
	\]
	For the first (resp. third) equality we use $ J = D_2^b \cdot I $ (resp. $ J' = D_2^b \cdot I' $) and the second follows by $ ( J, b ) \sim ( J', b ) $.
	Therefore $ \Sing ( I, b ) = \Sing ( I', b ) $.
	After a permissible blow-up $ \pi \colon \widetilde{ Z } \to \Spec ( R ) $ the transform $ ( \widetilde{I}, \nu_2 b ) $ of $ ( I, \nu_2 b ) $ is determined by $ I \cO_{\widetilde{ Z }} = H^{\nu_2 b } \widetilde{I} $, where $ H $ denotes the ideal sheaf of the exceptional divisor.
	For the transform of $ J $ we have $ J \cO_{\widetilde{ Z }} = H^b \widetilde{J} =   H^{ (1 - \nu_2) b } \widetilde{D_2^b} \cdot H^{\nu_2 b } \widetilde{I} $.
	($ \widetilde{D_2^b}  $ denotes the transform of $ D_2^b $).
	Thus the situation is the same as before the blow-up, $ \widetilde{J} =   \widetilde{D_2^b} \widetilde{I} $ and this is also true for $ J' $ and $ I' $.
	Together we get the desired equivalence $  ( I, \nu_2 b ) \sim  ( I', \nu_2 b ) $.
\end{proof}

\begin{Thm}[{Theorem~\ref{MainThm:nuPurelyPoly}}]
\label{Prop:ThmA2ndB}
	Let $ r \in \IZ_+ $. 
	Using the notation of Construction~\ref{Constr:nu},
	consider $ ( \cF_r, \cE_r  ) $ on $ N_{r-1} = V ( y_1, \ldots, y_{r-1} ) $.
	Let $ ( u , y_r) = ( u_1, \ldots, u_e; y_r ) $ be the remaining part of the {\RSP} for the regular local ring $ R $,
	where $ y_r $ is the parameter defining the next hypersurface (in $ N_{r-1} $) with maximal contact in Construction~\ref{Constr:nu}.
	Then:
	\begin{enumerate}[(1)]
		\item 
		$ \mu_{ r + 1 } = \delta (  \Delta( \cF_r , u, y_r) ) $.
		
		\item 
		Write $ \cE_{r,N_r}  := \{ (H_1, d_1),  \ldots, (H_j, d_j) \} $ for the exceptional data on $ N_r = V(y_1, \ldots, y_r) $ of $ ( \cF_r, \cE_{r,N_r} ) $.
		The entry $  \nu_{ r + 1 } = \nu_{ r + 1 } ( x ) $ of $ \inv_X(x) $ coincides with $ \nu ( \cF_r , \cE_{r,N_r} , u ) = \delta (  \Delta( \cF_r , u, y_r) ) - \sum_{ i = 1 }^j d_i .
		$
	\end{enumerate}
\end{Thm}

\begin{proof}
	This follows from the definition of $ \mu_{ r + 1 }, \mu_{ r + 1, H  } $ and $ \nu_{ r + 1 } $ 
	(see Construction~\ref{Constr:nu}).
\end{proof}

In conclusion, the invariant $ \nu_{ r + 1 } = \nu_{ r + 1 } ( x ) $ can be completely determined by only considering polyhedra. 
By Proposition~\ref{Prop:nuinvariant},  $ \nu ( \cF_r , \cE_{r,N_r} , u ) $,
and thus $ \nu_{ r + 1 }  $, is independent of the choice of a representative as {\IEH} as well as of the choice of $ ( y ) $ (for fixed $ ( u) $).
Moreover, equivalent {\IPsH} determine the same invariant $ \inv_X ( x ) $ by Lemma \ref{Lem:ObjectsEquiv},
i.e.,
$ \inv_X ( x) $ is an invariant of the \IEH.

\begin{Rk}
	If $ \cE_s( x ) = \emptyset $ for some $ s $, then the remaining process coincides with the one described in Section~\ref{Sec:NoExc} without exceptional divisors.
\end{Rk}

\begin{Obs}[\em Behavior of the polyhedron]
	\label{Obs:BehaviorPoly}
Let us see what happens to our polyhedra in each step of the general process.
Let $ r \in \IZ_+ $ with $ 0 < r \leq n $ and let $ e = n - r $.

\noindent 
\textbf{\emph{From $\boldsymbol{\cG_r }$ to $\boldsymbol{\cF_r  = \cG_r  \cap ( E^r ,1)} $:}}
In this step we add 
$
	 \left( E^r , 1 \right) = \bigcap_{ H \in E^r } \, (g_H, 1),
$
where $ g_H $ denotes a local generator of $ H $.
Recall that $ s_r = \#E^r  $.
By construction $ E^r  \subseteq \cE_{r-1} $ has only normal crossings with $ N_{r-1} $.
Thus we can choose the {\RSP} $ ( u ) = ( u_1,\ldots, u_{e+1} ) $ for $ N_{r-1} $ such that for all $ H \in E^r $ the local generator is $ g_H = u_i $ for $ i \in I_r:= \{i_1, \ldots, i_{s_r} \} \subseteq \{1, \ldots, e+1 \} $.
(In fact, we can choose the {\RSP} such that the analogous condition holds for every $ H \in \cE_{r-1} $).

Adding the exceptional components $ ( E^r,1 ) $ corresponds to adding the set of points 
$
\{ 
( \delta_{ \alpha i })_{ \alpha \in \{1,\ldots, e + 1 \} }  \mid i \in I_r 
\} $
to the points determining the polyhedron $ \Delta^N ( \cG_r , u ) $,
where $\delta_{\alpha i}$ denotes the usual Kronecker delta.

\noindent 
\textbf{\emph{From $\boldsymbol{\cF_r}$ to $\boldsymbol{\cH_r}$:}}
Suppose $ \cF_r = ( f_1, b_1) \cap \ldots \cap ( f_k, b_k) $, then there exists at least one $ i \in \{ 1, \ldots, k \} $ such that $ b_i = \ord_M ( f_i ) $.
We assume \WLOG that $ y_r := u_{ e + 1 } $ has maximal contact with $ \cF_r $.
Hence in this step we project the polyhedron $ \Delta^N ( \cF_r;  u_1, \ldots, u_e, u_{ e +1 } ) \subset \IR_{\geq 0}^{ e + 1} $ from the point $ ( 0, \ldots,0,1) \in \IR^{e +1 }_{\geq 0} $ to $ \IR_{\geq 0}^e $.
The resulting polyhedron is
$
 \Delta ( \cF_r;  u_1, \ldots, u_e; y_r ) = \Delta^N ( \cH_r;  u_1, \ldots, u_e ) \subset \IR_{\geq 0}^{ e }.
$

\noindent 
\textbf{\emph{From $\boldsymbol{\cH_r }$ to $\boldsymbol{\cG_{r+1}}$:}}
Suppose $ \cH_r = ( h_1, b_1) \cap \ldots \cap ( h_p, b_p) $.
The last step consists of three smaller steps.
We determine $ D_{r+1} $ and write each $ (h_i,b_i) $ as $ h_i = D_{r+1}^{b_i} \cdot g_i $.
Set 
\[
	\widetilde{\cH_r}  := \bigcap_{ i = 1 }^p \;  (g_i, b_i) 
\quad 
\mbox{ and }
\quad 
	\widetilde{\cG_{r+1}}  :=  \bigcap_{ i = 1 }^p  \;(g_i, b_i  \nu_{r+1}) .
\]
Note that $ \cG_{r+1} = 	\widetilde{\cG_{r+1}} \cap ( D_{r+1}, 1 - \nu_{r+1}) $.
The smaller steps are: 
\begin{enumerate}[(1)]
	\item \label{It:cases}
	\textbf{\emph{From $\boldsymbol{\cH_r}$ to $\boldsymbol{\widetilde{\cH_r}}$:}} 
	 		Since $ N_r $ and $ \cE_r $ have simultaneously only normal crossings, we can choose the coordinates $ (u_1, \ldots, u_e) $ of $ N_r $ such that for all $ H \in \cE_r $ the local generator is $ g_H = u_\ell $ for some $ \ell \in \{1, \ldots, e \} $.
	 		In this situation we set $ \mu_{r+1,\ell} := \mu_{r+1,H} $.
 			Put 
 			$$ 
 				I_r :=\left\{\, \ell_1,\, \ldots,\, \ell_{m_r} \,\right\} := \left\{ \, \ell \in \{1, \ldots, e \} \;\mid \; \mu_{r+1,\ell} \neq 0 \, \right\} \subseteq \{1, \ldots, e \} .
 			$$
 			We denote by $ \tau_r : \IR^e \to \IR^e $ the translation in the negative direction by the vector 
 			$$
 				\dBu{w}{r} := 
 						\bigg( 	0,\,\ldots,\,
							 	0,\,\underset{\displaystyle \ell_1}{\underset{\displaystyle \uparrow}{\mu_{r+1,1}}},\,0,\,\ldots,\,
								0,\,\underset{\displaystyle \ell_2}{\underset{\displaystyle \uparrow}{\mu_{r+1,2}}},\,0,\,\ldots,\,
								0,\,\underset{\displaystyle \ell_{m_r}}{\underset{\displaystyle \uparrow}{\mu_{r+1,m_r}}},\,0,\,\ldots,\,
								0 \bigg),
 			$$
 			where the marking $ \ell_\iota $ indicates that the $ \ell_\iota$-th entry is specified.
 			In other words, a point $ v \in \IR^e $ is sent to $ \tau_r ( v ) = v - \dBu{w}{r}  $.
			Then we have for the Newton polyhedra
			$ 
				\tau_r \left(\, \Delta^N ( \cH_r , u ) \, \right) = \Delta^N ( \widetilde{\cH_r} , u ) \subseteq \IR^e_{\geq 0}. 
			$
	\item
	\textbf{\emph{From $\boldsymbol{\widetilde{\cH_r}}$ to $\boldsymbol{\widetilde{\cG_{r+1}}}$:}} 
			In this step we multiply each point of the polyhedron $ \Delta^N ( \widetilde{\cH_r}  , u ) $ by the factor $ \nu_{r+1}^{-1} $ and get $ \Delta^N ( \widetilde{\cG_{r + 1 }} , u ) $.
 	\item
 	\textbf{\emph{From $\boldsymbol{\widetilde{\cG_{r+1}}}$ to $\boldsymbol{\cG_{r+1}}$:}}
 			The last step is similar to ``\emph{From $\cG_r $ to $\cF_r $}''.
 			By definition $ \cG_{r+1} = \widetilde{\cG_{r+1}} \cap (D_{r+1}, 1 - \nu_{r+1}) $.
 			Thus we add the points associated to 
 			\[
 				\left(D_{r+1} = \prod_{ H \in \cE_{r} } g_H^{ \mu_{r+1,H} },\, 1 - \nu_{r+1} \right),
 			\]
 			to the points determining $ \Delta^N ( \widetilde{\cG_{r + 1 }}  , u ) $, 
 			where $ g_H $ is a local generator of $ H \in \cE_{r} $. 
 			As in \eqref{It:cases}, we choose $ ( u ) = (u_1,\ldots, u_e) $ such that for all $ H \in \cE_r $ the local generator is $ g_H = u_\ell $ for some $ \ell \in \{1, \ldots, e \} $.
 			Again we set $ \mu_{r+1,\ell} := \mu_{r+1,H} $
 			and 
 			$
 				I_r :=\left\{ \ell_1, \ldots, \ell_{m_r} \right\} := \left\{  \ell \in \{1, \ldots, e \} \mid  \mu_{r+1,\ell} \neq 0  \right\} \subseteq \{1, \ldots, e \} .
 			$
 			Then $ (D_{r+1}, 1 - \nu_{r+1})  $ yields in $ \Delta^N ( \cG_{r + 1 } , u ) $ the point
 			$$
 			\bigg( 	0,\,\ldots,\,
							 	0,\underset{\displaystyle \ell_1}{\underset{\displaystyle \uparrow}{\frac{\mu_{r+1,1}}{1 - \nu_{r+1}} }},0,\,\ldots,\,
								0,\underset{\displaystyle \ell_2}{\underset{\displaystyle \uparrow}{\frac{\mu_{r+1,2}}{1 - \nu_{r+1}} }},0,\,\ldots,\,
								0,\underset{\displaystyle \ell_{m_r}}{\underset{\displaystyle \uparrow}{\frac{\mu_{r+1,m_r}}{1 - \nu_{r+1}} }},0,\,\ldots,\,
								0 \bigg).
 			$$
\end{enumerate}
\rightline{\em (End of Observeration~\ref{Obs:BehaviorPoly})}
\end{Obs}

%
%
%
%
%
%
%
%
%
%
%
%
%
%

The construction of the invariant of Bierstone and Milman takes several small steps.
Therefore it is hard to formulate a step-by-step result on the behavior of the generators of the ideal $ J $ as we did in Proposition \ref{Prop:MainThmNoExc} and Observation \ref{Obs:quasihomogeneous}.
We end the present article by explaining how the procedure becomes easier in certain good situations.

\begin{Not}
	Fix $ r \in \IZ_+ $.
	Using the notation of Construction~\ref{Constr:nu},
	let $ \cI_r \in \{ \cG_r , \cF_r , \cH_{r + 1}  \} $ and $ s \in \IZ_+ $, $ s < r $.
	\begin{enumerate}[(1)]
		\item We define the {\em $ \cG_s $-part of $ \cI_r $} to be the part of $ \cI_r $ which is by the construction coming from $ \cG_s $.
		\item	By the {\em $ \dBu{E}{s} $-part} (resp. {\em $ \dBu{D}{s} $-part) of $ \cI_r $} we denote the part which occurred by adding $ E^s, \ldots, E^r $ (resp. $ D_{s + 1} , \ldots, D_r $).
	\end{enumerate}	 
	For $ \cI_r = \cG_r $ we neglect $ E^r  $ in the definition of the $ \dBu{E}{s} $-part, because it has not been added yet.
	If $ s = 1 $, then we speak also of the $ \cG $-part (resp. $ E $-part, resp. $ D $-part) of $ \cI_r $.
\end{Not}

\begin{Obs}[\emph{Big steps with the exceptional part} $ { (E^q  ,1) } $]
\label{Obs:BigEstep}
	In the definition of $ \cF_1  $ we add the old exceptional components $ (E^1 , 1 ) $ to $ \cG_1  $.
	This enables us to make sometimes more than one step in Construction \ref{Constr:nu}:
	Since $ E^1 $ is a simple normal crossing divisor on $ N_0  = \Spec(R) $,
	we can choose a {\RSP} $ ( T ) = (T_1, \ldots, T_n ) $ for $ R $ such that every $ H \in E^1 $ is locally given by some $ T_\ell = 0 $ for $ \ell \in \{1,\ldots,n\} $,
	say $ E^1 $ is given by $ (  T_{\ell_1},\, \ldots,\, T_{\ell_{s_1}} ) $.
	Suppose $ s_1  \geq 1 $. 
	Set $ ( z ) = ( z_1, \,\ldots,\, z_{s_1} ) = ( T_{\ell_1},\, \ldots,\, T_{\ell_{s_1}} ) $.
	Then $ V ( z ) $ has maximal contact  with $ \cF_1 $.
	So this is a possible choice for the first $ s_1 $ steps in definition of $ \nu_i = \nu_i ( x ) $.
	Since $ \cE_1 $ and $ E^1 $ have simultaneously only normal crossings, we can require the additional property on $ ( T ) $ that $ \cE_1 $ is given by $ ( T_{ m_1 }, \ldots, T_{ m_p} ) $, 
	where $ T_\iota  \neq T_\rho  $ for $ \iota \in \{ m_1, \ldots, m_p \} $ and $ \rho \in \{ \ell_1, \ldots, \ell_{s_1 } \} $. 
	Thus we get for every $ i \in \{ 2,\ldots, s_1 \} $ 
	(If $ s_1 = 1 $ this set is empty):
	\begin{enumerate}[(1)]
		\item 	$ \mu_{i,H} = 0 $ for every $ H \in \cE_i $, thus $ D_i = 1 $ and
		\item	$ \nu_i = \mu_i = 1 $.
	\end{enumerate}

	Set $ d := s_1 $.
	In the procedure we added  $ E^2, \ldots, E^{d} \subset \cE_1  $.
	Recall that $ s_q = \# E^q  $ for $ q \in \{ 1, \ldots, d \} $ and $ \cE_d = E_j \setminus \bigcup_{ q = 1 }^{ d } E^q $,
	where $ E_j $ is the set of exceptional divisors at this step of the resolution process, see \eqref{eq:bigseq}.
	If 
	$ 
		s_1 + \cdots + s_d - d \geq 1 
	$
	, then $ D_{ d + 1 } = 1 $, $ \nu_{ d + 1 } = \mu_{ d + 1 } = 1 $ and we can choose the next maximal contact in the $ E $-part of $ \cH_d  $.

	\begin{center} \emph{{\bf \em Convention.} We choose the maximal contact variables in the $ E $-part until we reach the step $ r > d $ where $ s_1 + \ldots + s_r - r = 0 $.}
		\end{center} 

	This means the $ E $-part of $ \cH_r $ is empty.
	Recall that $ \cH_r $ determines $ \nu_{ r + 1 } $.
	As above it follows for every $ i \in \{ 2,\ldots, r \}$:
	\begin{enumerate}[(1)]
		\item[{(1')}] 	$ \mu_{i,H} = 0$ for every $ H \in \cE_i$, thus $ D_i = 1 $ and
		\item[{(2')}]	$ \nu_i = \mu_i = 1 $.
	\end{enumerate}
	In particular, $ \cH_r $ is only given by the $ \cG_1 $-part.
	{\em This means, $ \cH_r $ is the coefficient pair of $ \cG_1  $ \wrt $ V(z_1, \ldots, z_r ) $}.

	In general, we cannot assume $ s_1 > 0 $. 
	Set
	\[
		d := \min \left\{\, q \in \IZ_+ \;|\; s_q \neq 0 \,\right\}.
	\]
	Then $ E^d  \neq \emptyset $ and $ \cF_d = \cG_d \cap ( E^d  , 1 ) $.
	We choose the maximal contact $ V (z_{d}) $ such that there is some $ H \in E^d $ which is locally given by $ V (z_{d}) $.

	If $ s_d \geq 2 $, then the $ \dBu{E}{d} $-part of $ \cH_{d} $ is non-empty. 
	This implies $ \nu_{ d + 1 }  = \mu_{ d + 1}  = 1 $.
	In the next step of the procedure we multiply the assigned numbers by $ \nu_{ d + 1 } = 1 $, thus $ \cG_{ d + 1 }  = \cH_d $ and then we add $ E^{ d + 1 } $ in order to obtain $ \cF_{ d + 1 } $.
	We choose the maximal contact in the $ \dBu{E}{d} $-part and so on. 
	This continues until we are at the step
	\[
		r:= \min \left\{\, \ell \in \IZ_+ \mid \ell \geq d \ \mbox{ and } \ s_d + \cdots + s_\ell - ( \ell-d+1) = 0 \,\right\}.
	\]
\end{Obs}
	
Putting everything together yields:
	
\begin{Prop}
\label{Prop:EpartResult}
	Let $ d, r \in \IZ_+ $ be as above. 
	For every $ i \in \{ d + 1, \ldots , r \} $ we get
	\begin{enumerate}[(1)]
		\item[(i)] 		$ \mu_{i,H} = 0 $ for every $ H \in \cE_i$, thus $ D_i  = 1 $ and
		\item[(ii)]		$ \nu_i = \mu_i= 1 $,
		\item[(iii)] 	the $ \dBu{E}{d} $-part of $ \cH_r  $ (and $ \cG_{r + 1} $) is empty,
		\item[(iv)] 	hence $ \cH_r  $ is the coefficient pair of $ \cG_d  $ \wrt $ V(z_{ d }, \ldots, z_r ) $,
		\[
			\cH_r  =
			\ID ( \cG_d; u ; z_{ d }, \ldots, z_r  )		,
		\]
		 and $ \mu_{ r + 1 } = \delta (  \Delta ( \cG_d; u ; z_{ d }, \ldots, z_r  )) $, where $ ( u ) $ denotes a system extending $ (z) $ to a {\RSP} for $ R $.
	\end{enumerate}
	Further, $ \nu_{ r + 1 } $  is given by $ \mu_{ r + 1 } $ and the assigned numbers in the exceptional data of $ \cH_r  $.
\end{Prop}

Note that Proposition~\ref{Prop:EpartResult} depends on the convention that we choose the maximal contact first in the $ E $-part of the given {\IPH}.
	
\begin{Obs}[\emph{Big steps if} $ { D_q  = 1 } $]
\label{Obs:BigDstep}
	Set
	\[
		d := \min \left\{\, q \in \IZ_+ \;|\; D_q  = 1  \,\right\}
	\hspace{10pt}
	\mbox{ and }
	\hspace{10pt}
		r:= \min \left\{\, \ell \in \IZ_+ \mid \ell > d \ \mbox{ and } \  D_q \neq 1 \,\right\}.
	\]
	Since $ D_d  = 1 $, we have $ \cG_d  = \cH_{ d - 1 } $.
	If $ s_d = \# E^d  = 0 $, then the next step is as without exceptional divisors.
	On the other hand, if $ s_d \geq 1 $, then we can apply Observation \ref{Obs:BigEstep} until the $ E $-part is empty.
	
	This works until we come to $ \cH_{ r - 1 }  $ where $ D_r  \neq 1 $.
	By the convention of choosing first the exceptional components in the $ E $-part, the $ \dBu{E}{d} $-part of $ \cH_{ r - 1 }  $ is empty.
	This implies that $ \cH_{ r - 1 } $ is only given by the $ \cG_d $-part.
	(But it is not necessarily the coefficient pair of $ \cG_d  $ \wrt $ V( z_d, \ldots, z_{ r - 1 } ) $, because maybe not all $ \nu_i  $ are equal $ 1 $ for $  d < i < r $;
	nevertheless the situation is similar to Observation \ref{Obs:quasihomogeneous}
	 --- see  \cite[Remark~3.5.6]{BerndThesis}). 
	
	We modify $ \cH_{ r - 1 }  $ as described in Construction \ref{Constr:nu} 
	(factor out $ D_r  $ and then add $ ( D_r  , 1 - \nu_{ r } ) $)
	and obtain $ \cG_r $.
	
	If $ \mu_{ r }  = 1 $, then 
	$ (D_r , 1 - \nu_r) \simHx \bigcap_H\, (g_H, 1) $,
	where the intersection is over those $ H \in \cE_r $ with $ \mu_{r,H} \neq 0 $ and $ g_H $ denotes a local generator of $ H $.
	Then the same procedure as in the Observation~\ref{Obs:BigEstep} can be applied:
	First we choose the maximal contact only in the part coming from $ D_r  $ and after that we consider the $ \dBu{E}{r} $-part.
	
	We can apply this until we get to the point, where $ D_{r'}  \neq 1 $ and $ \mu_{r'}  > 1 $.
	Then we have to apply the full procedure to construct $ \nu_{ r' } $ and we go back to the beginning of this observation. 
\end{Obs}

%
%
%
%
%
%
%
%
%
%
%
%
%
%


\end{document}